\documentclass[12pt]{amsart}
\usepackage{amsmath,amscd,amsthm,amsfonts, amssymb,amsxtra}
\usepackage[all]{xy} \SelectTips{cm}{}
\newtheorem{thm}{Theorem}[section]  
\newtheorem*{un-no-thm}{Theorem}
\newtheorem{cor}[thm]{Corollary}     
\newtheorem{lem}[thm]{Lemma}         
\newtheorem{prop}[thm]{Proposition}  
\newtheorem{add}[thm]{Addendum}
\newtheorem{conjecture}[thm]{Conjecture}
\newtheorem{bigthm}{Theorem}
\newtheorem{bigcor}[bigthm]{Corollary}

\newtheorem{bigadd}[bigthm]{Addendum}

\theoremstyle{definition}
\newtheorem{defn}[thm]{Definition}   

\theoremstyle{definition}

\theoremstyle{definition}
\theoremstyle{remark}
\newtheorem{rem}[thm]{Remark}
\newtheorem{rems}[thm]{Remarks}
\newtheorem*{ques}{Question}

\newtheorem*{intro-rem}{Remark}
\newtheorem*{intro-rems}{Remarks}

\newtheorem{ex}[thm]{Example}

\DeclareMathOperator{\cd}{cd}

\DeclareMathOperator*{\colim}{colim}
\DeclareMathOperator{\id}{id}
\DeclareMathOperator{\maps}{map}
\DeclareMathOperator{\RO}{RO}

\DeclareMathOperator{\interior}{int}
\DeclareMathOperator{\self}{end}
\DeclareMathOperator{\secs}{sec}
\DeclareMathOperator{\Sp}{Sp}

\begin{document}
\title{Homotopical Intersection Theory, II: equivariance}
\date{\today} 
\author{John R. Klein} 
\address{Wayne State University,
Detroit, MI 48202} 
\email{klein@math.wayne.edu} 
\author{Bruce Williams} 
\address{University of Notre Dame, Notre Dame, IN 46556}
\email{williams.4@nd.edu}
\begin{abstract} This paper is a sequel to \cite{klein-williams}.
We develop here an  intersection theory 
for manifolds equipped with an action of a finite group.
As in \cite{klein-williams}, our approach will be homotopy
theoretic, enabling us to circumvent the 
specter of equivariant transversality.

We give applications of our theory to embedding problems,
equivariant fixed point problems and the study of
periodic points of self maps.
\end{abstract}
\thanks{The first author is partially supported by the 
National Science Foundation}
\maketitle
\setlength{\parindent}{15pt}
\setlength{\parskip}{1pt plus 0pt minus 1pt}
\def\Top{\bold T\bold o \bold p}
\def\wTop{\text{\rm w}\bold T}
\def\wT{\text{\rm w}\bold T}
\def\vo{\varOmega}
\def\vs{\varSigma}
\def\smsh{\wedge}
\def\flush{\flushpar}
\def\dbslash{/\!\! /}
\def\:{\colon}
\def\Bbb{\mathbb}
\def\bold{\mathbf}
\def\cal{\mathcal}
\def\orb{\cal O}
\def\hoP{\text{\rm ho}P}

\setcounter{tocdepth}{1}
\tableofcontents
\addcontentsline{file}{sec_unit}{entry}

\section{Introduction \label{intro}}
\subsection*{Intersection problems}
Suppose $N$ is a compact smooth manifold equipped with a closed
submanifold $Q\subset N$. An {\it intersection problem} for $(N,Q)$
consists of a map $f\: P \to N$, where $P$ is a closed manifold. A
{\it solution} to the problem consists of a homotopy of $f$ to a map
$g$ satisfying $g(P) \cap Q = \emptyset$. We depict the situation by
$$
\SelectTips{cm}{}
\xymatrix{
& N - Q \ar[d] \\
P \ar[r]_f\ar@{..>}[ur]
& N\, ,
}
$$ in which we seek to find the dotted arrow making the diagram
homotopy commute. One also has a version of the above when
$P$ has a boundary whose image under $f$ is disjoint from $Q$. 
We then require the deformation of $f$ to hold the
boundary fixed. Let $i_Q\: Q\subset N$ be the inclusion.
We will often denote the  data by $(f,i_Q)$.

In \cite{klein-williams}, we produced an obstruction
$\chi(f)$ living in a certain bordism group whose vanishing is
necessary for the existence of a solution. Furthermore, the
obstruction was shown to be sufficient in the range $p \le 2n - 2q -
3$, where $\dim N = n, \dim Q = q$ and $\dim P = p$. We also gave a
version of the obstruction for families.
\medskip

Here, we will consider {\it equivariant} intersection problems. Suppose
$G$ is a finite group and the above manifolds are equipped with
smooth $G$-actions. 
\medskip

In the equivariant setting, $i_Q\:Q \subset N$ is a $G$-submanifold
and $f\: P \to N$ is an equivariant map. We now seek a deformation
of $f$ through $G$-maps to an 
equivariant map whose image is disjoint from $Q$.

The partial answers we will give to such questions
are phrased in terms of isotropy data.
If $X$ is a $G$-space, we let 
$$
{\cal I}(G;X)
$$ 
denote the conjugacy classes of subgroups of $G$ which appear
as stabilizer groups of points of $X$.

\subsection*{Indexing functions}
An {\it indexing function} $\phi_\bullet$ on a $G$-space $X$ assigns to a
subgroup $H\subset G$ a locally constant function $\phi_H$ with domain
$X^H$, the fixed point set of $H$ acting on $X$, and codomain given by
the extended integers $\Bbb Z\cup \pm\infty$.  It is also required to
be conjugation invariant: if $K = gHg^{-1}$ and $h\:X^H \to X^K$ is
the homeomorphism $x\mapsto gx$, then $\phi_H = \phi_K \circ h$.

If $\psi_\bullet$ is another indexing function on $X$,
and $H \subset G$ is a subgroup,
we write 
$$
\phi_H \le \psi_H
$$ 
if 
$\phi_H(x) \le \psi_H(x)$ for
all $x\in X^H$. If $\phi_H \le \psi_H$ for all $H$, then we write
$\phi_\bullet \le \psi_\bullet$.

Here are some examples:

\subsubsection*{Dimension}  If $M$ is a locally smooth $G$-manifold, then for
any subgroup $H \subset G$ the components of the fixed point set $M^H$
are manifolds \cite[Ch.\ 4]{Bredon}.  The dimensions of the components
can vary. If $x\in M^H$, then the dimension of the component
containing $x$ defines a locally constant function $m^H$. The
collection $m^\bullet := \{m^H\}_{H\subset G}$ is called the {\it
dimension function} of $M$.  If $M^H$ is empty, our convention is to
set $m^H = -\infty$.

\subsubsection*{Codimension}
Let $i_Q\:Q \subset N$ be as above.
Let $$\cd_\bullet(i_Q)$$ be the indexing function
on $Q$
in which $\cd_H(i_Q)(x)$ is the dimension of
the normal space
to the embedding $Q^H \subset N^H$ at $x\in Q^H$ (if $Q^H$ is empty but $N^H$ isn't, 
our convention is to set $\cd_H(i_Q) = +\infty$).


\subsubsection*{Pullback}
Suppose $f\:X\to Y$ is a $G$-map. 
Given an indexing function $\alpha_\bullet$ on $Y$, we
obtain an indexing function $f^*\alpha_\bullet$ on
$X$ which is given by $f^*\alpha_H(x) = \alpha_H(f(x))$.

\subsubsection*{Pushforward} Given $f\: X\to Y$ as above, 
 let $\beta_\bullet$ be an indexing function on $X$. 
 For $y \in Y$ we let $[y]$ denote the associated path component.
 Let $I_{f,y}$ be the set of those $[x]$ for which $[f(x)] = [y]$. That is,
 $I_{f,y}$ is the inverse image of $f_*\: \pi_0(X) \to \pi_0(Y)$ at $[y]$.
 
 Define
an indexing function $f_!\beta_\bullet$ on $Y$ by the rule
\[
f_!\beta_H(y) = 
\begin{cases} \inf_{I_{f,y}} \beta_H(x) \quad &\text{ if $I_{f,y}$ is nonempty, } \\
 \infty  \quad &\text{ otherwise.}
 \end{cases}
\]
Note $f_!f^*\alpha_\bullet \ge \alpha_\bullet$, with equality holding when
$f_*$ is a surjection, whereas $f^*f_!\beta_\bullet \le \beta_\bullet$
with equality holding when $f_*$ is an injection.

\subsection*{Stable intersections}
Just as in the unequivariant case,
equivariant intersections can be removed when the codimension is sufficiently
large.  The equivariant intersection problem $(f,i_Q)$
is said to be {\it stable} if 
$$
p^H \,\,  \le \,\,  f^*(i_Q)_!\cd_H(i_Q) - 1\,  
$$
for every $(H) \in {\cal I}(G;P)$ (Roughly, this means the dimension of the transverse intersection
of $f(P^H)$ and $Q^H$ is negative).

If the intersection problem is stable,
one can use
elementary equivariant obstruction theory to show $f$
equivariantly deforms off of $Q$, yielding a solution.

\subsection*{A ``cohomological'' result} 
Our first main result gives a complete obstruction
to solving equivariant intersection problems in the equivariant
{\it metastable} range. The obstruction lies
in the cohomology of $P$ with coefficients in a certain
{\it parametrized equivariant spectrum} over $N$.  A reader who is not
familiar with this technology should consult \S\ref{proof-cohom}.

\begin{bigthm} \label{cohom} To an equivariant intersection problem $(f,i_Q)$,
there is a naive parametrized $G$-spectrum ${\cal E}(i_Q)$ over 
$N$, which is constructed from the inclusion $i_Q \: Q\subset N$, 
and an obstruction
$$
e_G(f) \in H^0_G(P;{\cal E}(i_Q))
$$
which vanishes when the intersection problem has a solution.

Conversely, if $e_G(f)=0$ and
$$
p^H \le 2f^*(i_Q)_{!}\cd_H(i_Q)  - 3
$$
for all $(H) \in {\cal I}(G;P)$,
then the intersection problem has a solution.
\end{bigthm}

\begin{intro-rem} Theorem \ref{cohom} is an equivariant version
of \cite[{Cor.\ 3.5}]{klein-williams}.
The word {\it naive} is used here indicate that the parametrized
spectrum is indexed over a trivial universe; the
equivariant cohomology theory of the theorem 
is therefore of ``Bredon type.''

The inequalities of Theorem 
 \ref{cohom} define the {\it equivariant metastable range.} 
When $G$ is the trivial group, one has 
the sole inequality $2p \le 2n -2q-3$, which is just
the unequivariant metastable range (cf.\ \cite{klein-williams}). 
\end{intro-rem}

\subsection*{Homotopical equivariant bordism}
Since naive equivariant cohomology theories are not indexed 
over representations, they are not fully ``stable.'' 
From our viewpoint, a crucial deficiency 
of naive theories is their lack of Poincar\'e duality.

To get around this,
we impose additional conditions to get a more
tractible invariant residing in a theory
which does possess Poincar\'e duality. We will map the  
equivariant cohomology theory of Theorem \ref{cohom}
into a similarly defined $\RO(G)$-graded one. 
The additional constraints will insure the map is injective.
Applying duality to our pushed-forward invariant, we obtain
another invariant living in $\RO(G)$-graded homology.
We then identify the homology theory with the
homotopical $G$-bordism groups of a certain $G$-space.

To a $G$-space $X$ equipped with real $G$-vector bundle $\xi$,
one has an associated {\it equivariant Thom spectrum}
$$
X^\xi \, ,
$$ 
whose spaces $X^\xi_V$ are indexed by representations $V$ ranging over a
complete $G$-universe ${\cal U}$ (compare \cite[Chap.\
XV]{May}). Here, $X^\xi_V$ denotes the Thom space of $\xi \oplus
V$. Equivalently, $X^\xi$ is the equivariant suspension spectrum
of the Thom space of $\xi$. More generally, $X^\xi$ 
is defined whenever $\xi$ is an {\it virtual} $G$-bundle over $X$
(see \cite[Ch.\ 9]{LMS} for details).

For a virtual $G$-representation $\alpha = V-W$, the {\it homotopical $G$-bordism group}
of $(X,\xi)$ in degree $\alpha$ is given by
$$
\Omega^G_\alpha(X;\xi) \,\,  :=  \,\, 
\colim_{U} 
  [S^{V+U},X^\xi_{W+U}]^G \, ,
$$ 
where $[S^{V+U},X^\xi_{W+U}]^G$ denotes the homotopy classes of based
$G$-maps $S^{V+U} \to X^\xi_{W+U}$ (in which $S^{V+U}$ is 
the one point compactification of the direct sum of $V$ and $U$),
and the colimit is indexed over the finite dimensional
subrepresentations $U$ of ${\cal U}$ using the partial ordering 
defined by inclusion. Actually, we will only need 
consider the case when $\alpha = 0$ is the trivial representation of rank zero.

\begin{intro-rems}(1).
There is a related object, ${\cal N}^G_\alpha(X;\xi)$, called the {\it
geometric bordism group} of $(X,\xi)$. It is generated by
$G$-manifolds $M$ equipped with $G$-map $u\:M \to X$ and a stable
$G$-bundle isomorphism $$u^*\xi\oplus \tau_M \oplus \epsilon_W \,\, \cong \,\,\epsilon_V\, ,$$ 
where $\epsilon_V$ denotes the $G$-bundle whose total space is 
$X \times V$.

The Pontryagin-Thom construction defines a homomorphism 
$$
{\cal N}^G_\alpha(X;\xi) \to \Omega^G_\alpha(X;\xi)\, . 
$$ 
In constrast with the
unequivariant case, this map can fail to be an isomorphism because of
the lack of equivariant transversality (see \cite{Petrie}, 
\cite{CW}, \cite[{p.\ 156}]{May}).
\smallskip

{\flushleft (2).} When $\xi$ is a $G$-vector bundle (not virtual),
then $X^\xi$ is the equivariant suspension spectrum of the Thom space
of $\xi$. In particular, when $\xi$ is trivial of rank zero, we get
$\Sigma^\infty_G (X_+)$, the equivariant suspension spectrum of 
$X\amalg *$.  In this case, the map from the
equivariant geometric bordism group to the homotopical one is an
isomorphism \cite{Hauschild}, \cite{Kos}.  The $k$-th homotopy group of 
$\Sigma^\infty_G (X_+)$
coincides with $\Omega^{G,\text{fr}}_k(X)$, the $k$-dimensional
equivariant framed bordism group of $X$.
\smallskip

{\flushleft (3).} When $G = e$ is the trivial group, and $\xi$ has virtual
rank $n$, $\Omega^e_0(X;\xi) = \Omega_0(X;\xi)$ is the bordism group generated by maps
$\alpha\:M \to X$, with $M$ a compact $n$-manifold, together with a
(stable) isomorphism $\alpha^*\xi$ with the stable normal bundle of
$M$. Note the indexing convention used here is different from the one
of \cite{klein-williams} (the latter implicitly ignored the rank of $\xi$
but indicated the dimension of the manifolds in the degree of the bordism group; thus 
the group $\Omega_n(X;\xi)$ of \cite{klein-williams} coincides with
the current $\Omega_0(X;\xi)$).
\end{intro-rems}
\medskip

We now specialize to the equivariant bordism groups
arising from intersection problems.  Given an
equivariant intersection problem $(f,i_Q)$, define
$$
E(f,i_Q)
$$ 
to be the {\it homotopy fiber product} (a.k.a.\ homotopy pullback)
of $f$ and $i_Q$. A point in $E(f,i_Q)$ is a triple $(x,\lambda,y)$ in
which $x\in P$, $y \in Q$ and $\lambda\: [0,1]\to N$ is a path such
that $\lambda(0) = f(x)$ and $\lambda(1) = y$. There is an evident
action of $G$ on $E(f,i_Q)$.

There are forgetful maps $j_P\:E(f,i_Q) \to P$ and $j_Q\: E(f,i_Q) \to Q$,
both equivariant. There is also an equivariant
map $j_N \: E(f,i_Q) \to N$ given by $(x,\lambda,y) \mapsto \lambda(1/2)$.
Using these, we obtain an equivariant virtual bundle over $E(f,i_Q)$ by
$$
\xi  \,\,  :=  \,\, j_N^*\tau_N - j_Q^*\tau_Q - j_P^*\tau_P\, .
$$
If $Q\subset N$ is held fixed, then $\xi$ is completely determined
by $f\:P \to N$.


\subsection*{A ``homological'' result} 

\begin{bigthm} \label{equi-int} 
Given an equivariant intersection problem $(f,i_Q)$, there is
an invariant
$$ 
\chi_G(f) \in \Omega^G_0(E(f,i_Q);\xi)
$$
which vanishes when $f$ is equivariantly
homotopic to a map whose image is disjoint from $Q$.

Conversely, assume $\chi_G(f) = 0$ and 
\begin{itemize}
\item for each
$(H) \in {\cal I}(G;P)$, we have
$$
p^H \le 2f^*(i_Q)_{!}\cd_H(i_Q)  - 2 \, ;
$$ 
\item for each $(H) \in {\cal I}(G;P)$ and each proper subgroup
$K \subsetneq H$, we have
$$
p^H \le f^*(i_Q)_{!}\cd_K(i_Q) - 2\, .
$$
\end{itemize}
Then $f$ is equivariantly homotopic to a map whose image is disjoint
from $Q$.
\end{bigthm}

\begin{intro-rems} (1). The assignment
$
f\mapsto \chi_G(f)
$ 
is a {\it global section} of a locally constant sheaf over the
equivariant mapping space $\maps(P,N)^G$.  The stalk of this sheaf at
$f$ is $\Omega^G_0(E(f,i_Q);\xi)$.  This explains the sense in which
$\chi_G(f)$ is an invariant: an equivariant homotopy from $f$ to another map
$f'\: P \to N$ gives rise to an isomorphism of stalks over $f$ and $f'$, and the
isomorphism transfers $\chi_G(f)$ to $\chi_G(f')$.
\smallskip

{\flushleft (2).} The second set of inequalities of Theorem \ref{equi-int}
can be regarded as a {\it gap condition}.
\smallskip

{\flushleft (3).} An advantage that Theorem \ref{equi-int} enjoys
over Theorem \ref{cohom} is that the obstruction group appearing in the former is
is defined directly in terms of the maps $f\: P \to N$ and $Q \to N$. 
It turns out that obstruction group of Theorem \ref{cohom} is defined in terms of $f$ the map $N-Q\to N$, which is not as easy to identify in terms of the input data. Furthermore, 
the equivariant bordism group appearing in 
Theorem \ref{equi-int} arises from a Thom spectrum indexed over a complete universe, 
so more machinery is at hand for the purpose of making calculations (see \cite{May}).
\end{intro-rems}

\subsection*{Boundary conditions}  
There is also a version of
Theorem \ref{equi-int} when $N$ is compact, possibly with boundary, and
$P$ is compact with boundary $\partial P\ne \emptyset$ satisfying
$f(\partial P) \cap Q = \emptyset$.
In this instance one seeks an equivariant deformation of $f$, fixed on 
$\partial P$, to a new map whose image is disjoint from $Q$. 

\begin{bigadd} \label{boundary} Theorem $\ref{equi-int}$ also holds
when $P$ and $N$ are compact manifolds with boundary, where it is assumed
$f(\partial P) \cap Q = \emptyset$ and $Q$ is embedded in the interior of $N$.
\end{bigadd}

\subsection*{Sparse isotropy}
When the action of $G$ on $P$ has few isotropy types, the inequalities
in Theorem \ref{equi-int} unravel somewhat.

\subsubsection*{Free actions}
Suppose the action of $G$ on $P$ is free.
Then the trivial group is the only isotropy group and the inequalities
of  Theorem \ref{equi-int} reduce to a single inequality
$$
p \le 2(n-q) - 3 = 2n-2q-3\, .
$$
Furthermore, the
equivariant bordism group of Theorem \ref{equi-int} is isomorphic
to the unequivariant bordism group
$$
\Omega_0(EG \times_G E(f,i_Q);\id_{EG}\times_G\xi)\, ,
$$
where $EG \times_G E(f,i_Q)$ is the Borel construction. 
This bordism group is generated by maps $u\:M \to EG \times_G E(f,i_Q)$
together with a stable isomorphism $\nu_M \cong u^*(\id_{EG}\times_G\xi)$,
where $M$ has dimension $p+q-n$ and $\nu_M$ denotes the stable normal bundle.
The identification of these groups is obtained using a transfer construction 
(we omit the details).

\subsubsection*{Trivial actions}
If $P$ has a trivial $G$-action, then the only isotropy group is
$G$. In this instance, $P$ has image in $N^G$ and the intersection
problem becomes an unequivariant one, involving the map $f\:P\to N^G$
and the submanifold $Q^G \subset N^G$.  Assume for simplicity that
$N^G$ and $Q^G$ are connected.  Then by \cite{klein-williams}, the
intersection problem admits a solution when $\chi(f) \in
\Omega_0(E(f,i_{Q^G});\xi)$ is trivial and $p \le 2n^G -2q^G -3$.

\subsubsection*{Prime order groups} Let $G$ be a cyclic
group of prime order. By the above, we can assume both the trivial group
and $G$ appear as stabilizer groups. Then $\emptyset \neq P^G \subsetneq P$.

For simplicity, assume $Q^G$ and $N^G$ are connected.
Then the first set of inequalities of Theorem \ref{equi-int} becomes
$$
p \le 2n-2q-3\, , \qquad p^G \le 2n^G-2q^G -3\, ,
$$
and the second set amounts to the single inequality
$$
p^G \le n-q-2\, .
$$

\subsection*{Local intersection theory}
Suppose an equivariant intersection problem $(f,i_Q)$ has been
partially solved in the following sense: there is $G$-subspace $U
\subset P$ such that $f(U)$ is disjoint from $Q$.  One can then ask
whether the solution extends to a larger subspace of $P$.  
A {\it local equivariant intersection problem} amounts to these data.
A systematic approach to such questions provided by 
the isotropy stratification of $P$.

\subsubsection*{The isotropy stratification}
The relation of {\it subconjugacy} describes a partial ordering ${\cal
I}(G;P)$: we will write
$$
(H) < (K)
$$ 
if $K$ is properly subconjugate to $H$.  We then choose a total ordering
which is compatible with the partial ordering.  Let
$$
(H_1) < (H_{2}) < \cdots <(H_\ell)
$$
be the maximal chain coming from the total ordering of ${\cal I}(G;P)$.

Let $P_i\subset P$ be the set of points $x$ having stabilizer group
$G_x$ in which $(G_x) \le (H_i)$.  Then we have a filtration of
$G$-spaces
$$
\emptyset = P_0 \subsetneq P_1 \subsetneq P_2 \subsetneq \cdots\subsetneq  P_\ell = P \, ,
$$
where each inclusion $P_i \subset P_{i+1}$ possesses the equivariant 
homotopy extension property (cf.\ \cite{Davis}, \cite{Illman}).

\subsubsection*{The local obstruction} 
Suppose $(f,i_Q)$ is an equivariant intersection problem 
with $f(P_{i-1}) \cap Q = \emptyset$ for some $i \ge 1$. We seek a
deformation of $f$ relative to $P_{i-1}$ to a new map $f'$ such that
$f'(P_i) \cap Q = \emptyset$. The map $f'$ is then a
solution to the local problem.

Let $H$ be a representative of $(H_i)$ and let $f_H \: P_H \to N$ denote
the restriction of $f$ to $P_H$.  The Weyl group $W(H) = N(H)/H$ acts
on $P^H$ and freely on $P_H$.  Let ${}_{H}\xi$ be the virtual $W(H)$-bundle
over $E(f_H,i_Q)$ defined by 
$$
j_N^*\tau_N - j_Q^*\tau_Q - j_{P_H}^*\tau_{P_H}\, .
$$

\begin{bigthm} \label{local} There is an invariant
$$
\chi^i_G(f) \in \Omega_0^{W(H)}(E(f_H,i_Q);{}_{H}\xi)) 
$$ 
which is trivial when the local problem at $P_i$ relative to
$P_{i-1}$ can be solved.

Conversely, assume $\chi^i_G(f) = 0$ and
$$
p^H \le 2f^*(i_Q)_{!}\cd_H(i_Q)  - 3 
$$ 
for $(H)= (H_i)$. Then 
the local problem admits a solution. 
\end{bigthm}

\subsubsection*{Descent} The global invariant $\chi_G(f)$ 
is an assemblage of all the local invariants.  
Although the local invariants may contain more information, 
they can fail to provide a solution to the global
question.  To address this point, we will give criteria for deciding
when the vanishing of the global invariants implies the vanishing of
the local ones.  In combination with Theorem \ref{local} the criteria
yield a kind of {\it descent} theory for equivariant intersection
problems.

Let $H \in {\cal I}(G;P)$ be and consider the inclusion
$$
P_{H} \subset P^{H}.
$$
The corresponding inclusion
$E(P_H,Q) \subset E(P^H,Q)$ by $t_H$.
The map $f^H\: P^H \to N$ will denote the restriction
of $f$ to $P^H$. Define a virtual $W(H)$-bundle ${}^H\xi$ over
$E(f^{H},i_Q)$ by $j_N^*\tau_N - j_Q^*\tau_Q - j_{P^H}^*\tau_{P^H}$.
Since the pullback of ${}^H\!\xi$ along $t_H$ is ${}_H\xi$,
we get an induced homomorphism
$$ 
(t_H)_*\: \Omega^{W(H)}_0(E(f_{H},i_Q);{}_H\xi) \to 
\Omega^{W(H)}_0(E(f^{H},i_Q);{}^H\!\xi) \, .
$$

\begin{bigthm}[``Global-to-Local''] \label{global-to-local} Assume
\begin{itemize}
\item 
$f(P_{i-1}) \cap Q = \emptyset$ 
for some $i \ge 1$ (so $\chi_G^i(f)$ is defined).
\item $(t_H)_*$ is  injective for  $(H)= (H_i)$.
\end{itemize}
Then $\chi_G(f) = 0$ implies 
$\chi^i_G(f)= 0$. 
\end{bigthm}

\begin{bigcor}[``Descent''] \label{descent}
Let $(f,i_Q)$ be an equivariant intersection problem.
Assume 
\begin{itemize} 
\item $\chi_G(f) = 0$,
\item  $(t_H)_*$ is  injective,
\item 
$p^H \le 2f^*(i_Q)_{!}\cd_H(i_Q)  - 3$,
\end{itemize}
for every $(H) \in {\cal I}(G;P)$. Then 
there is an equivariant deformation of $f$
to a map whose image is disjoint from $Q$.
\end{bigcor}

\subsection*{Applications}

\subsubsection*{Embeddings}
Suppose   $f\: P \to N$ is a smooth immersion. Equipping
$P$ with a Riemannian metric, we identify
the total space 
of the unit tangent disk bundle of $P$ with a compact tubular neighborhood
of the diagonal $\Delta_P \subset P \times P$. With respect
to this identification, the involution of $P \times P$ corresponds
to the one on the tangent bundle that maps a tangent vector to its negative.
Let $S(2)$ be the total space of the 
unit spherical tangent bundle of $P$, and let
$P(2)$ be the effect of deleting the interior of the tubular neighborhood
from $P \times P$. Then $(P(2),S(2))$ is a free 
$\Bbb Z_2$-manifold with boundary.

If we rescale the metric, then $f {\times} f$ 
determines an equivariant map
$$
(f(2),f(2)_{|S(2)})\:(P(2),S(2)) \to (N^{\times 2},N^{\times 2} - \Delta_N)\, ,
$$
which yields relative $\Bbb Z_2$-equivariant intersection 
problem with free domain. 
The fiber product $E(f(2),i_{\Delta_N})$ in this case coincides with the
space of triples $(x,\gamma,y)$ with $x,y \in P(2)$ and $\gamma$ a path
from $f(x)$ to $f(y)$. The involution is given by $(x,\gamma,y) \mapsto
(y,\bar \gamma,x)$, where $\bar\gamma(t) := \gamma(1-t)$.
We set $E'(f,f) := E(f(2),i_{\Delta_N})$.

Applying Addendum \ref{equi-int} and observing
the action is free, we have an obstruction
$$
\mu(f) \in \Omega_0(E{\Bbb Z}_2 \times_{\Bbb Z_2} 
E'(f,f);\id \times_{\Bbb Z_2}\xi)
$$
whose vanishing suffices for finding an equivariant deformation of 
$f(2)$, fixed on $S(2)$, to a map whose image is disjoint from $\Delta_N$,
provided $3p+3 \le 2n$. 

By a theorem of Haefliger \cite{Haefliger},
$f$ is regularly homotopic to an embedding in the metastable range $3p+3\le 2n$ if and only if the above equivariant intersection problem admits a solution. 
Consequently, 

\begin{bigcor}[compare {\cite[{th.\!\! 2.3}]{HQ}}]
If $f$ is regularly homotopic to an embedding, then
$\mu(f)$ is trivial.

Conversely, in the metastable range, the vanishing of 
$\mu(f)$ implies $f$ is regularly homotopic to an embedding.
\end{bigcor}

\subsubsection*{Equivariant fixed point theory}
Let $M$ be a closed smooth manifold equipped a smooth
action of a finite group $G$.
Let
$$
\maps^\flat(M,M)^G
$$ 
denote the space of
fixed point free $G$-maps from $M$ to itself. 
Equivariant fixed point theory studies the extent to which the inclusion
$$
\maps^\flat(M,M)^G \to \maps(M,M)^G
$$ 
is a surjection on path components.

For an equivariant self map $f\: M \to M$, let
$$
L_f M
$$ 
be the space of paths $\lambda \: [0,1] \to M$
satisfying the constraint $f(\lambda(0)) = \lambda(1)$. Then
$G$ acts on $L_f M$ pointwise. Let
$(L_f M)_+$ be the effect of adding a disjoint basepoint
to $L_fM$, and finally, let 
$$
 \Omega^{G,\text{fr}}_0(L_f M))
$$
be the $G$-equivariant framed bordism of $L_f M$ in dimension zero.

\begin{bigthm} \label{lefschetz} There is 
an invariant
$$
\ell_G(f) \in 
 \Omega^{G,\text{\rm fr}}_0(L_f M)
$$
which vanishes when $f$ is equivariantly homotopic to a
fixed point free map.

Conversely, assume $\ell_G(f) = 0$. If
\begin{itemize}
\item $m^H \ge 3$ for all $(H) \in {\cal I}(G;M)$.
\item $m^H \le m^K - 2$ 
  for proper inclusions $K \subsetneq H$ with $K, H \in {\cal I}(G;M)$,
\end{itemize}
then $f$ is equivariantly homotopic to a fixed point free map.
\end{bigthm}

\begin{intro-rems} (1). The above can be regarded as an equivariant
analog of a classical theorem of Wecken \cite{Wecken}.
\smallskip

{\flushleft (2).} A formula of tom Dieck splits
$\Omega^{G,\text{\rm fr}}_0(L_f M)$
into a direct sum of unequivariant framed bordism
groups indexed over the conjugacy classes of subgroups
of $G$. The summand corresponding to a conjugacy class $(H)$
is 
$$
\Omega^{\text{\rm fr}}_0(EW(H) \times_{W(H)} L_{f^H} M)\, ,
$$
where $EW(H) \times_{W(H)} L_{f^H} M$ is the Borel construction
of the Weyl group $W(H)$ acting on  $L_{f^H} M$
(see \cite{tD}, \cite{May}).
Consequently, $\ell_G(f)$ decomposes as a sum
of invariants indexed in the same way. 
We conjecture the Nielsen number $N(f^H)$ can
be computed from the projection of $\ell_G(f)$ onto
the displayed summand.
\smallskip

{\flushleft (3).} Our result bears close similarity to a theorem 
of Fadell and Wong \cite{Fadell-Wong} (see also
\cite{Ferrario},\cite{Weber}). Their result uses 
the Nielsen numbers $N(f^H)$ with $(H)\in
{\cal I}(G;M)$ in place of our $\ell_G(f)$.
\end{intro-rems}

\subsubsection*{Periodic Points} 
A fundamental problem in  discrete dynamics
is to   enumerate the periodic orbits of a self map
$f\: M\to M$, where $M$ is a closed manifold.

Let $n \ge 2$ be an integer.
A point $x\in M$
is said to be  {\it $n$-periodic} if $x$ is a
fixed point of the $n$-th iterate of $f$, i.e., $f^n(x) = x$.
The set of $n$-periodic points of $f$ is
denoted
$$
P_n(f)\, .
$$
The cyclic group $\Bbb Z_n$ acts on $P_n(f)$: if $t\in {\Bbb Z_n}$ is a
generator, then the action is defined by
$t\cdot x := f(x)$. 

The {\it homotopy $n$-periodic point set} of $f$
is the $\Bbb Z_n$-space
$$
\text{ho} P_n(f)
$$
consisting of $n$-tuples 
$$
(\lambda_1,\lambda_2,...,\lambda_n)\, ,
$$
in which $\lambda_i\:[0,1] \to M$ is a path and the data are
subject to the constraints
$$
 f(\lambda_{i+1}(0)) = \lambda_i(1))\, , \qquad i = 1,2,\dots
$$
Here we interpret the subscript $i$ as being taken modulo $n$.
The action of $\Bbb Z_n$ on $\text{ho}P_n(f)$ is 
given by cyclic permutation of factors.

There is a map $P_n(f) \to \hoP_n(f)$
 given by sending an $n$-periodic point $x$
to the $n$-tuple 
$$
(c_x,c_{f(x)},c_{f^2(x)},...,c_{f^{n-1}(x)})
$$
in which $c_x$ denotes the constant path with value $x$.
\medskip

For a self map $f\:M \to M$, as above, let
\[
\Omega_0^{\Bbb Z_n,\text{\rm fr}}(\hoP_n(f))
\] 
be the ${\Bbb Z}_n$-equivariant framed bordism 
group of $\hoP_n(f)$ in dimension $0$.

\begin{bigthm}\label{periodic} 
There is a homotopy theoretically defined invariant
$$
\ell_n(f) \in \Omega_0^{\Bbb Z_n,\text{\rm fr}}(\hoP_n(f))
$$
which is an obstruction to deforming $f$ to an $n$-periodic point
free self map.
\end{bigthm}

\begin{intro-rems}  
At the time of writing, we do not know the extent to which $\ell_n(f)$ is the
complete obstruction to making $f$ $n$-periodic point free.  When
$\dim M \ge 3$, Jezierski \cite{Jezierski} has shown the vanishing of
the Nielsen numbers $N(f^k)$ for all divisors $k|n$ implies $f$
is homotopic to an $n$-periodic point free map (here $f^k$ denotes
the $k$-fold composition of $f$ with itself). We conjecture that
$\ell_n(f)$ determines $N(f^k)$. 
\end{intro-rems}

\subsubsection*{Periodic points and the fundamental group}
Let $\pi$ be a group equipped with endomorphism 
$\rho\: \pi\to \pi$.
Consider the equivalence relation on $\pi$ generated by the
elementary relations
$$
x \,\sim \, gx\rho^n(g)^{-1}  \quad \text{and} \quad x\, \sim \, \rho(x)
$$
for  $x,g \in \pi$.
Let
$$
\pi_{\rho,n}
$$
be the set of equivalence classes.
Let
$$
\Bbb Z[\pi_{\rho,n}]
$$
denote the free abelian group with basis $\pi_{\rho,n}$. 

Let $f\: M \to M$ be a self map of a connected closed manifold $M$.
Fix a basepoint $*\in M$.
Choose a homotopy
class of path $[\alpha]$ from $*$ to $f(*)$.
Then $[\alpha]$ defines an isomorphism
$$
\pi_1(M,*) \cong \pi_1(M,f(*))\, .
$$
Furthermore, $f$ and $[\alpha]$ together define a homomorphism 
$$
\rho\: \pi_1(M,*) \overset {f_\sharp} \to \pi_1(M,f(*)) \cong \pi_1(M,*) \, .
$$
Let $\pi = \pi_1(M,*)$.

\begin{bigthm} \label{stalk} The data consisting
of the self map $f\: M \to M$, 
the choice of basepoint $*\in M$ and the homotopy 
class of path $[\alpha]$ from $*$ to $f(*)$
determine an isomorphism of abelian groups
$$
\Omega_0^{\Bbb Z_n,\text{\rm fr}}(\hoP_n(f)) \,\, \cong \,\,
\bigoplus_{k|n}  \Bbb Z[\pi_{\rho,k}]\, .
$$
With respect to this isomorphism, there is
a decomposition
$$
\ell_n(f)\,\,  = \,\,\underset {k|n}\oplus \, \ell_n^k(f)\, ,
$$ 
in which $\ell_n^k(f)\in \Bbb Z[\pi_{\rho,k}]$.
\end{bigthm}

\section{Preliminaries \label{prelim}}

\subsection*{$G$-Universes} 
The $G$-representations of this paper are assumed to come equipped
with a $G$-invariant inner product.  A {\it $G$-universe} ${\cal U}$
is a countably infinite dimensional real representation of $G$ which
contains the trivial representation and which contains infinitely many
copies of each of its finite dimensional subrepresentations.

We will be interested in two kinds of universes.  A {\it complete}
universe is one that contains infinitely many copies of
representatives for the irreducible representations of $G$ (in this
instance one can take $\cal U$ to be the countable direct sum of the
regular representation). A {\it trivial} universe contains only
trivial representations.

\subsection*{Spaces}
We work in the category of compactly generated topological spaces.
The empty space is $(-2)$-connected and every non-empty space is
$(-1)$-connected. A map $A \to B$ of spaces (with $B$ nonempty) is
$r$-connected if for any choice of basepoint in $B$, the
homotopy fiber with respect to this choice of basepoint is an
$(r{-}1)$-connected space In particular, any map $A \to B$ is
$(-1)$-connected.  A weak homotopy equivalence 
is an $\infty$-connected map.

\subsection*{$G$-spaces}
Let $G$ be a finite group. A {\it $G$-space} is
a space $X$ equipped with a left action of $G$. A map
of $G$-spaces is a $G$-equivariant map.

Let $T$ be a transitive $G$-set.
The {\it $T$-cell} of dimension $j$ is the $G$-space
$$
T \times D^j \, ,
$$
where $G$ acts diagonally with trivial action on $D^j$.

\begin{rem} If a choice of basepoint $t\in T$ is given, then
one has a preferred isomorphism $G/H \cong T$, where  
$H = G_t$ is the stabilizer of $t$. Given another choice
of basepoint $t'$, the stabilizer group $G_{t'}$ is conjugate
to $H$. We will call the conjugacy class $(H)$ the {\it type} of $T$. 
Two transitive $G$-sets are isomorphic if and only if
they have the same type.
\end{rem}

Given a $G$-map $f\:T \times S^{j-1} \to Y$, one may form 
$$
Y \cup_f (T \times D^j)\, .
$$
This is called an {\it $T$-cell attachment.} 
If $j = 0$, we interpret the above as a disjoint union.

A {\it relative $G$-cell complex} $(X,Y)$ is a pair in which
$X$ is obtained from $Y$ by iterated equivariant cell attachments
(where we allow $T$ to vary over different transitive $G$-sets;
the collection of attached cells is allowed to be a class). 
The order of  attachment defines a partial ordering on the 
collection of cells.
If this order is dimension preserving (i.e., no cell of dimension
$j$ is attached after a cell of dimension $j'$ when $j < j'$), then
$(X,Y)$ is a {\it relative $G$-CW complex}.
When the collection of such attachments is finite, one
says $(X,Y)$ is {\it finite}.
When $Y$ is the empty space, $X$ is a {\it $G$-cell complex} and
when the attachments are self-indexing, $X$ is a $G$-CW complex.

The {\it cellular dimension function} $d_\bullet$ for 
$(X,Y)$ is the indexing function whose value at $H$ is
the maximal dimension of the cells of type $(H)$ 
appearing in the collection of attached cells.
We set $d_H = -\infty$ if  $(X,Y)$ has no cells of type
$(H)$.

\begin{rem} Let $M$ be a closed smooth
$G$-manifold with dimension function $m^\bullet$. A result of Illman \cite{Illman} shows that
$M$ possesses
an equivariant triangulation. If $d_\bullet$ is the cellular dimension
function of this triangulation, then $d_H = m^H$ for all $H \in {\cal I}(G;M)$. 
\end{rem}

\subsection*{Quillen model structure}
Let $T(G)$ be the category of $G$-spaces.
A morphism $f\:X \to Y$ is a {\it weak equivalence} if
for every subgroup $H\subset G$ the induced map of fixed points
$$
f^H\: X^H \to Y^H
$$
is a weak homotopy equivalence. Similarly, a morphism
$f$ is a {\it fibration}
if $f^H$ is a Serre fibration for every $H$.
A morphism $f\: X\to Y$ is a {\it cofibration} if
there is a relative $G$-cell complex $(Z,X)$ such that $Y$ is a retract
of $Z$ relative to $X$.  

Let $R(G)$ be the category of {\it based} $G$-spaces.
A morphism $X\to Y$ of $R(G)$ is a weak equivalence, 
cofibration or fibration if and only if it is so when 
considered as a morphism of $T(G)$.

\begin{prop}[\cite{Dwyer-Kan_sing-real}, {\cite[Ch.\ VI \S5]{May}}]  
With respect to the above structure,
both $T(G)$ and $R(G)$ are Quillen model categories.
\end{prop}

\subsection*{Connectivity}
One says an indexing function
$r_\bullet$ is  a {\it connectivity function}
 for a $G$-space $Y$ if $Y^H$ is $r_H$-connected for $H \subset G$
a subgroup (if $Y^H$ is empty, we set $r_H = -2$). 
If $f\:Y \to Z$ is a morphism of $T(G)$, then a connectivity function
for $f$ is an indexing function $r_\bullet$
such that $f^H\:Y^H \to Z^H$  is an $r_H$-connected map
of spaces (one can always assume $r_H \ge -1$ 
since every map of spaces is at least $(-1)$-connected).

\begin{lem} \label{factorization} Let $Y \to Z$ 
be a fibration of $T(G)$ with connectivity function $r_\bullet$.
Suppose $(X,A)$ is a  relative $G$-cell complex
with cellular dimension function $d_\bullet$. Assume 
$d_H \le r_H$ for all subgroups $H \subset G$.
Then given a factorization
problem of the form 
$$
\xymatrix{
A \ar[r] \ar[d] & Y \ar[d] \\
X \ar[r]\ar@{..>}[ur] & Z
}
$$
we can find an equivariant lift $X \to Y$ such that
the diagram commutes.
\end{lem}

\begin{rem} The condition $d_H \le r_H$ is automatically
satisfied if no cells of type $(H)$ occur in $(X,A)$.
\end{rem}

\begin{proof}[Proof of Lemma \ref{factorization}]
The proof proceeds by induction on the equivariant cells which
are attached to $A$ to form $X$. 
The inductive step is
reduced to solving an equivariant lifting problem of the kind
$$
\xymatrix{
G/H \times S^{j-1} \ar[r] \ar[d] & Y \ar[d]^f \\
G/H \times D^j \ar[r] \ar@{..>}[ur] & Z\, ,\\
}
$$
where the horizontal maps are allowed to vary in their
equivariant homotopy class.

Now, a $G$-map $G/H \times U \to Z$
when $U$ has a trivial action is the same thing as specifying
a map $U \to Z^H$. This means the lifting problem
reduces to an unequivariant one of the form
$$
\xymatrix{
S^{j-1} \ar[r] \ar[d] & Y^H \ar[d]^{f^H} \\
D^j \ar[r]\ar@{..>}[ur]  & Z^H \\
}
$$
The latter lift exists because $f^H$ is $r_H$-connected and $j \le r_H$.
\end{proof}

\begin{cor}\label{factorize-homotopy} Consider the lifting problem
$$
\xymatrix{
A \ar[r] \ar[d] & Y \ar[d] \\
X \ar[r]\ar@{..>}[ur] & Z
}
$$
of $G$-spaces in which 
\begin{itemize}
\item $Y \to Z$ is a map 
with connectivity function $r_\bullet$,
\item $Y$ is cofibrant,
\item $(X,A)$ is
a relative $G$-cell complex with dimension function $d_\bullet$,
\item $d_H \le r_H$ for each subgroup $H \subset G$.
\end{itemize}
Then there is a $G$-map $X \to Y$ making the top triangle of
the diagram commute and the bottom triangle homotopy commute.
\end{cor}

\begin{proof} Factorize the map $Y \to Z$ as
$$
Y \to Y^\text{c} \to Z
$$
in which the map $Y \to Y^\text{c}$ is a cofibration and a weak equivalence
and the map  $Y^\text{c} \to Z$ is a fibration. Apply Lemma 
\ref{factorization}
to the diagram with $Y^\text{c}$ in place of $Y$. To get a map
$X \to Y^\text{c}$ making the diagram commute. 

Since every object is fibrant, the acyclic cofibration
$Y \to Y^{\text{c}}$ is a retract; let $r \: Y^{\text{c}} \to Y$
be a retraction.  Let $f\: X \to Y$ be the map $X \to Y^{\text{c}}$
followed by the retraction. Then $f$ satisfies
the conclusion stated in the corollary.
\end{proof}

\subsection*{Fiberwise $G$-spaces}
Fix a $G$-space $B$. A {\it $G$-space over $B$} is a
$G$-space $X$ equipped with $G$-map $X \to B$, usually denoted
$p_X$. A morphism
$X\to Y$ of $G$-spaces over $B$ is a $G$-map which commutes with
the structure maps $p_X$ and $p_{Y}$. Let
$$
T(B;G)
$$
be the category of $G$-spaces over $B$.

We also have a ``retractive'' version of this category, denoted
$$
R(B;G)\, .
$$
An object of the latter consists of a $G$-space $X$ and
maps $p_X\: X \to B$, $s_X \: B \to X$ such that $p_Xs_X = \text{id}_B$.
A morphism $X\to Y$ is an equivariant map compatible with both structure 
maps.

In either of these categories, a morphism $X \to Y$ is said to be a
weak equivalence/cofibration/fibration if it is so when
considered as a morphism of $T(G)$
by means of the forgetful functor. With respect to these
definitions, $T(B;G)$ and $R(B;G)$ are Quillen model
categories. 

An object $X$ in either of these categories is said
to be {\it $r_\bullet$-connected} if the structure map
$X \to B$ is $(r_\bullet+1)$-connected. A morphism
is said to be $r_\bullet$-connected if the underlying map
of $T(G)$ is.

The category $R(B;G)$ has {\it internal smash products}, constructed
as follows: let $X,Y \in R(B;G)$ be objects. Then
$$
X\smsh_B Y \in R(B;G)
$$
is the object given by the pushout of the diagram
$$
\begin{CD}  
         B  @<<<  X \cup_B Y @> \subset >>       X \times_B Y
\end{CD}
$$
where $X\times_B Y$ is the fiber product of $X$ and $Y$. 

Since $R(B;G)$ is a model category, one can form homotopy
classes of morphisms. If $X,Y\in  R(B;G)$, we let
$$
[X,Y]_{R(B;G)}
$$
denote the set of homotopy classes of morphisms. Recall 
the definition  requires us to replace $X$ by its cofibrant approximation
and $Y$ by its fibrant approximation.

\section{The proof of Theorem \ref{cohom} \label{proof-cohom}}

\subsection*{Unreduced fiberwise suspension}
Let $E\in T(B;G)$ be an object. 
The {\it unreduced fiberwise suspension} of $E$ over $B$ 
is the object $S_B E \in T(B;G)$ given by the double mapping
cylinder
$$
S_B E := B\times 0 \cup E\times [0,1] \cup B\times 1\, .
$$
The two evident inclusions $s_-,s_+\: B \to S_BE$ are 
morphisms of $T(B;G)$. 
Using $s_-$, we will consider $S_B E$ as an object of $R(B;G)$.

\subsection*{Obstruction to sectioning}
Let $$B^+$$ denote $B \amalg B$ considered as an object of
$R(B;G)$ using the left summand to define a section.
Then
$$
s := s_- \amalg s_+\: B^+\to S_B E
$$
is a morphism of $R(B;G)$. We consider the associated
homotopy class
$$
[s] \in [B^+,S_B E]_{R(B;G)} \, .
$$

The following proposition is an equivariant version of results of
Larmore (\cite[th.\ 4.2-4.3]{Larmore}; see also
 \cite[prop.\ 3.1]{klein-williams}).

\begin{prop} \label{Larmore} Assume $E\in T(B;G)$ is fibrant. If
$E \to B$ admits an equivariant section, then $[s]$ is
trivial. 

Conversely, assume 
\begin{itemize}
\item $[s]$ is trivial,
\item $B$ is a $G$-cell complex with dimension 
function $b_\bullet$, 
\item the object $E \in T(B;G)$ is $r_\bullet$-connected, and
\item  $b_\bullet \le 2r_\bullet + 1$.
\end{itemize}
Then
$E\to B$ admits an equivariant section.
\end{prop} 

\begin{proof} Let $\sigma\: B \to E$ be a section. Apply
the functor $S_B$ and note $S_BB = B \times [0,1]$.
We then get a map
$$
S_B \sigma \: B\times [0,1] \to S_B E
$$
which gives a homotopy from $s_-$ to $s_+$ through morphisms
of $T(B;G)$. This is the same thing as establishing the triviality
of $[s]$.

Conversely, 
the diagram
$$
\xymatrix{
E \ar[r]\ar[d] & B\ar[d]^{s_+} \\
B \ar[r]_{s_-} & S_B E
}
$$
is preferred homotopy commutative in the category $T(B;G)$. 
As a diagram of $G$-spaces it is a homotopy pushout.
Let $H \subset G$ be a subgroup. Taking $H$-fixed points, we obtain
a homotopy pushout
$$
\xymatrix{
E^H \ar[r]\ar[d] & B^H\ar[d]^{s^H_+} \\
B^H \ar[r]_{s^H_-} & S_{B^H} E^H
}
$$
in the category $R(B^H;e)$ where $e$ is the trivial group.
Since $E$ is an $r_\bullet$-connected object, the map
$E^H \to B^H$ is $(r_H + 1)$-connected.

By the Blakers-Massey theorem (see e.g., \cite[p.\ 309]{Good_calc2}), 
the second diagram is 
$(2r_H+1)$-cartesian. Consequently, the first diagram
is $(2r_\bullet + 1)$-cartesian. Let $P$ denote the
homotopy inverse limit of the diagram
$$
\begin{CD}
B @> s_- >> S_B E @< s_+ << B\, .
\end{CD}
$$
Then we conclude the map $E \to P$ is $(2r_\bullet + 1)$-cartesian.
If we assume $[s] = 0$, then the map $P \to B$ admits a section
up to homotopy (using the universal property of the homotopy pullback). 
By the assumptions on $B$ and Corollary \ref{factorize-homotopy}, 
the $G$-map  $E\to B$ admits a 
section up to homotopy. Since $E$ is fibrant,
this homotopy section can be converted into a strict section.
\end{proof}

\subsection*{Naive stabilization}
The {\it reduced fiberwise suspension} $\Sigma_B E$ of an object
$E\in R(B;G)$ is given by considering $E$ as an object
of $T(B;G)$, taking its unreduced fiberwise suspension 
$S_B E$ and taking the pushout of the diagram
$$
B \leftarrow S_B B \to S_B E
$$
where  $S_B B \to S_B E$ arises by applying
$S_B$ to the structure map $B\to E$.

A {\it naive parametrized $G$-spectrum} ${\cal E}$ is a collection of objects
$$
{\cal E}_n \in R(B;G)
$$
equipped with maps $\Sigma_B {\cal E}_n \to {\cal E}_{n+1}$.

\begin{ex} Let $Y \in R(B;G)$ be an object. 
Its naive parametrized suspension spectrum $\Sigma_B^\infty Y$
has $n$-th object $\Sigma^n_B Y$, the $n$-th iterated fiberwise
suspension of $Y$.
\end{ex}

\begin{defn} Let $X \in T(B;G)$ be an object.
The zeroth {\it cohomology} of $X$ with coefficients in ${\cal E}$
is the abelian group given by 
$$
H^0_G(X;{\cal E}) \,\, := \,\, \colim_{n\to \infty}
[\Sigma^n_B X^+,{\cal E}_n]_{R(B;G)}
$$
where $X^+ = X\amalg B$ and the maps in the 
colimit  arise from the structure maps of ${\cal E}$.
\end{defn}

\begin{rem} Assuming the maps ${\cal E_n} \to B$ are fibrations,
one can take the pullbacks $f^*{\cal E_n} \to X$. These form
a naive $G$-spectrum over $X$, 
and an unraveling of the definitions gives
$$
H^0_G(X;f^*{\cal E}) = H^0_G(X;{\cal E}) \, .
$$
\end{rem}

\begin{defn} Let $X, E\in T(B;G)$ be objects
with $E$ fibrant and $X$ cofibrant. Let
$f\: X \to B$ be the structure map. 
Let
$$
e(f,E) \in H^0_G(X;\Sigma_B^\infty S_B E)
$$
be the class defined by the map 
$$
\begin{CD}
X^+ @> f^+ >> B^+ @> s >> S_B E\, .
\end{CD}
$$
\end{defn}

\begin{prop} Let $X,E$ and $f$ be as above. If
$E \to B$ admits an equivariant section along $f$, 
then $e(f,E)$ is trivial.

Conversely, assume
\begin{itemize}
\item $e(f,E)$ is trivial,
\item $X$ is a $G$-cell complex with dimension 
function $k_\bullet$, 
\item $E \in T(B;G)$ is $r_\bullet$-connected, and
\item  $k_\bullet \le 2r_\bullet + 1$.
\end{itemize}
Then $E \to B$ admits an equivariant section along $f$.
\end{prop}

\begin{proof}  By Proposition \ref{Larmore}, it will be enough to prove
the maps
$$
\Sigma_B \: [\Sigma^n_B X^+,\Sigma^n_B S_B E]_{R(B;G)} \to 
[\Sigma^{n+1}_B X^+,\Sigma^{n+1}_B S_B E]_{R(B;G)}
$$
are isomorphisms in the stated range. We will do this when
$n = 0$. The case $n >0$ is similar.

We have a map
$$
E \to \Omega_B \Sigma_B E
$$
which is adjoint to the identity. By Corollary \ref{factorize-homotopy}, 
it will be enough to show this morphism is  
$2r_\bullet + 1$-connected. Let $H \subset G$ be a subgroup
and consider the map
$$
E^H \to \Omega_{B^H} \Sigma_{B^H} E^H
$$
of $R(B^H;e)$. If $b\in B^H$ is any point, we have an induced map of
fibers 
$$
E_b^H \to \Omega\Sigma E_b^H \, .
$$
Since $E_b^H$ is $r_H$-connected, the Freudenthal suspension
implies the last map is  $(2r_H+1)$-connected. We infer that the map
$E^H \to \Omega_B \Sigma_B E^H$ is  $(2r_H+1)$-connected, which is 
what we needed to show. 
\end{proof}

\subsection*{Proof of Theorem \ref{cohom}}

\begin{lem} \label{connectivity}
The map $(N-Q) \to N$ is  $(i_Q)_{!}\cd_H(i_Q)-1)$-connected.
\end{lem}

\begin{proof} Let $H \subset G$ be a subgroup. Then $(N-Q)^H = N^H - Q^H$,
and we need to compute the connectivity of the inclusion
$$
N^H - Q^H \to N^H \, .
$$
This will be done using transversality.

Consider a map of pairs 
$$
\gamma\:(K,A) \to (N^H,N^H-Q^H)
$$
$(K,A) = (D^j,S^{j-1})$ or $(S^j,\emptyset)$.
We can assume $\gamma$ is transverse to $Q^H$.
If $y \in \gamma(K) \cap Q^H$, then it must be the case
$j \ge \cd_H(i_Q)_H(y)$. Therefore, $\gamma(K)$
is disjoint from $Q^H$ whenever $j <\cd_H(i_Q)_H(y)$
for all $y\in Q^H$. This is equivalent to requiring
$j < (i_Q)_{!}\cd_H(i_Q)$, so the conclusion follows.
\end{proof}

Let $$(N-Q) \to E \to N$$ be the effect of factorizing
$N-Q \to N$ as an acyclic cofibration followed by a fibration.
By Lemma \ref{connectivity}, $E \in T(N;G)$ 
is  an $(i_Q)_{!}\cd_H(i_Q)-2)$-connected object.

Since $N-Q$ is cofibrant, it will suffice to show 
$E \to N$ admits a section along $f$. 

We set ${\cal E}(i_Q)$ equal to the naive fiberwise suspension spectrum
$$
\Sigma^\infty_N S_N E 
$$
and
$$
e_G(f) := e(f,E) \in H^0_G(P;{\cal E}(i_Q))\, .
$$
By the first part of Proposition \ref{Larmore}, 
if $E$ admits a section along $f$, then $e_G(f)$ is trivial.

Conversely, assume $e_G(f)$ is trivial.
Then 
$$
e(\id_P,f^*E) \in H^0_G(P;f^*{\cal E}(i_Q))
$$
is also trivial. One easily checks $f^*E$ is an 
$(f^*(i_Q)_{!}\cd_\bullet(i_Q)-2)$-connected object.
By the second part of Proposition \ref{Larmore}, 
the fibration $f^*E \to P$ admits a section
when
$$
p_\bullet \le 2f^*(i_Q)_{!}\cd_\bullet(i_Q) - 3.
$$
This completes the proof of Theorem \ref{cohom}.

\section{Naive versus Equivariant stabilization \label{naive-versus-equi}}

\subsection*{The unfibered case}
If $Y \in R(G) = R(*;G)$ is a cofibrant object, we define
$$
Q_GY  = \colim_{V} \Omega^V\Sigma^V Y\, ,
$$
where $V$ ranges over the finite dimensional subrepresentations of a
complete $G$-universe ${\cal U}$ partially ordered
with respect to inclusion, and
$\Omega^V \Sigma^V Y$ is the space of
unequivariant based maps $S^V \to S^V\smsh Y$, where
$S^V$ is the one point compactification of $V$. 
This is a $G$-space by conjugating maps by 
group elements.

Consider the natural $G$-map
$$
QY \to Q_G Y\, .
$$
\begin{prop} \label{universe_change}
Assume $Y$ has connectivity function $r_\bullet$.
Then the map $QY \to Q_G Y$ is $s_\bullet$-connected,
where
$$
s_H = \inf_{K \subsetneq H} r_K\,  .
$$
\end{prop}

\begin{proof} Let $H \subset G$ be a subgroup.
We must show the map of fixed points
$$
Q(Y^H) = (QY)^H \to  (Q_G Y)^H \, .
$$
is $s_H$-connected.
By the tom Dieck splitting (\cite[p.\ 203, Th.\ 1.3]{May},
\cite[th.\ 7.7]{tD}), 
$$
(Q_GY)^H \,\, \simeq \,\, \prod_{(K)} QEW(K)_+\smsh_{W(K)}Y^K\, ,
$$
where $(K)$ varies over the conjugacy classes of subgroups
of $H$ and $W(K)$ denotes the Weyl group. The factor
corresponding to $(K) = (H)$ gives the inclusion
$Q(Y^H) \to (Q_G Y)^H$. As  $QEW(K)_+\smsh_{W(K)}Y^H$
is $r_K$-connected, 
it follows the inclusion
is
$(\inf_{K \subsetneq H} r_K)$-connected.
\end{proof}

\subsection*{Equivariant stabilization}
Let $V$ be a finite dimensional $G$-representation equipped with 
invariant inner product. We let $D(V)$
be its unit disk and $S(V)$ its unit sphere. 

Let $X \in T(B;G)$ be an object.
The {\it unreduced
$V$-suspension} of $X$ over $B$ is the object
$S^V_B X$  given by
$$
S(V) \times B \cup_{S(V) \times E} D(V) \times Y \, .
$$
Note the case of the trivial representation $V = \Bbb R$
recovers $S_B X$.

If $Y \in R(B;G)$ is an object, then
a reduced version of the construction is given by
$$
\Sigma^V_B Y =  B \cup_{D(V) \times B}  S^V_B Y \, .
$$
The {\it fiberwise $V$-loops} of $Y \in R(B;G)$ is the 
object
$$
\Omega^V_B E
$$
given by the space of pairs $(b,\lambda)$ in
which $b\in B$ and $\lambda \: S^V \to p^{-1}(b)$ is a based map.
The action of $g\in G$ on $(b,\lambda)$ is given by 
$(gb,g\lambda)$. 

Then $(\Sigma_B^V,\Omega_B^V)$ is an adjoint functor pair.

\begin{defn} For an object $Y \in R(B;G)$, define
$$
Q_B^G Y \,\, := \colim_{V} \Omega_B^V \Sigma_B^V Y \, ,
$$
where the colimit is indexed over the finite dimensional subrepresentations of
a complete $G$-universe ${\cal U}$.
\end{defn}

\begin{prop} \label{fib-equi-versus-naive} Assume $Y \in R(B;G)$ is fibrant and cofibrant,
with connectivity function $r_\bullet$. Then
$$
Q_B Y \to Q_B^G Y
$$
is $s_\bullet$-connected, where
$$
s_H := \inf_{K \subsetneq H} r_K\, .
$$
\end{prop}

\begin{proof} Let $H  \subset G$ be a subgroup. Then the $H$-fixed points
of $Q_B Y$ is $Q_{B^H} Y^H$. Consider the evident map
$$
Q_{B^H} Y^H \to (Q_B^G Y)^H \, .
$$
Let $b \in B^H$ be a point. Then the associated map of homotopy fibers
at $b$ is identified with
$$
Q Y_b^H \to (Q_G Y_b)^H\, ,
$$
where $Y_b$ is the fiber of $Y^H \to B^H$ at $b$.
By Proposition \ref{universe_change}, the map of homotopy fibers is $s_H$-connected.
We conclude $Q_B Y^H \to (Q_B^G Y)^H$ is
also $s_H$-connected.
\end{proof}

\section{Parametrized $G$-spectra over a complete universe \label{fibered-spectra}}
Let ${\cal U}$ be a complete $G$-universe. 
A {\it (parametrized) $G$-spectrum} ${\cal E}$ over $B$ indexed on ${\cal U}$
is a collection of objects
$$
{\cal E}_V \in R(B;G)
$$
indexed over the finite dimensional subrepresentations $V$ of ${\cal U}$
together with maps
$$
\Sigma_B^{V^\perp} {\cal E}_V \to {\cal E}_{W}
$$
for $V \subset W$, where $V^\perp$ is the orthogonal complement
of $V$ in $W$.

\begin{ex} Let $X \in R(B;G)$ be an object. The fiberwise equivariant
suspension spectrum of $X$, denoted $\Sigma^{\infty,G}_B X$ has $V$-th space
$$
Q_B^G(\Sigma^V_B X) \, .
$$
\end{ex}

These give rise to unreduced $\RO(G)$-graded cohomology
theories on $T(B;G)$.
In order to get the details of the construction right, it is
helpful to know a Quillen model structure is lurking in
the background. For this exposition, it will
suffice to explain what the weak equivalences, fibrant
and cofibrant objects are in this model structure. 
The reference for this material is the book \cite{May-Sig}.

One says ${\cal E}$ is {\it fibrant} if each of the adjoint maps
${\cal E}_W \to \Omega^W_B {\cal E}_{V\oplus W}$ is a weak 
equivalence of $R(B;G)$ and moreover, each of the maps
${\cal E}_W \to B$ is a fibration of $R(B;G)$. Any object
${\cal E}$ can be converted into a fibrant object ${\cal E}^\text{f}$
by a natural construction, called fibrant approximation.

A map ${\cal E} \to {\cal E'}$ is given by compatible maps
${\cal E}_V \to {\cal E'}_V$. A map is a {\it weak equivalence} if
after applying fibrant approximation, it becomes an equivalence
at each $V$. An object ${\cal E}$ is {\it cofibrant} if it is the retract
of an object which is obtained from the zero object by attaching
cells. Any ${\cal E}$ can be functorially replaced by a cofibrant object within
its weak homotopy type; this is called cofibrant approximation.

\subsection*{Cohomology and homology} 
Let ${\cal E}$ be as above, and assume it is both fibrant and cofibrant. 
Let $X \in T(B;G)$ be a cofibrant object.
The equivariant {\it cohomology} of $X$ with coefficients in $\cal E$ 
is the $\RO(G)$-graded theory on $T(B;G)$, denoted
$$
h^\bullet_G(X;\cal E)\, ,
$$
and defined as follows:
if $\alpha = V - W$ is a virtual representation, we set 
$$
h^\alpha_G(X;\cal E) := [S^W\smsh X^+,\cal E_V]_{R(B;G)}
$$
Similarly, 
the {\it homology} of $X$ with coefficients in ${\cal E}$,
denoted
$$
h_\bullet^G(X;{\cal E})\, ,
$$
is defined by
$$
h_\alpha^G(X;{\cal E}):= \colim_{U} 
[S^{V+U},({\cal E}_{W+U} \smsh_B X^+)/B]_{R(*;G)} \, ,
$$
where 
$({\cal E}_{W+U} \smsh_B X^+)/B$
is the effect of taking the mapping cone of the section 
$B \to {\cal E}_{W+U} \smsh_B X^+$.

Using fibrant and cofibrant approximation,
the above extends in a straightforward way to the case
of all objects $X$ and all ${\cal E}$ a $G$-spectrum over $B$
(we omit the details).

\begin{rem} \label{Po_notation} Here is an alternative approach 
to the above, based on \cite{PoHu}. A map of $G$-spaces 
$f: X \to Y$  induces a pullback functor
$$
f^*\: R(Y;G) \to R(X;G)
$$
given by $Z \mapsto Z\times_Y X$, with evident
structure map.

The functor $f^*$ has a right adjoint $f_*$ given by 
$$
T\mapsto \underline{\secs}_Y(X\to T) \, ,
$$ 
where
$\underline{\secs}_Y(X\to T)$ has total space 
$$
\{(y,s)|\, y\in Y, s\:X_y\to T_y\}\, .
$$
Here, $X_y$ denotes the fiber of $X \to Y$ at $y$,
$T_y$ is the fiber of the composite $T \to X \to Y$
at $y$
and $s\: X_y \to T_y$ is a based (unequivariant) map.
 
The functor $f^*$ also admits a left adjoint, denoted
$f_\sharp$, which is defined by 
$$
T \mapsto T\cup_X Y \, .
$$

Let $\Sp(Y;G)$ be the
category of $G$-spectra over $Y$.
If we make these constructions levelwise, we
obtain functors 
$$
f_*,f_\sharp\:\Sp(X;G) \to \Sp(Y;G) \, .
$$
Now take $f\: X\to *$ to be the constant map to a point,
and replace these functors by their derived versions
(using the Quillen model structure). Let ${\cal E}$ be
a fibred $G$-spectrum over $B$. If $X\in T(B;G)$ is an
object with structure morphism $p_X\:X\to B$,
we can take the (derived) pullback $p_X^*{\cal E}$, 
which is a $G$-spectrum over $X$. 
Then the $\RO(G)$-graded homotopy groups of the $G$-spectra
$$
f_*{p_X^*\cal E} \quad \text{and } \quad f_\sharp{p_X^*\cal E}
$$
yield the above cohomology and homology theories.
\end{rem}

\section{Poincar\'e Duality \label{duality-section}}

\subsection*{The orientation bundle}
Let $M$ be a $G$-manifold and $TM$ its tangent bundle.
Let $S^\tau \in R(M;G)$
defined by taking the fiberwise one point compactification of
$TM \to M$ (the section $M \to S^\tau$ is given by the zero section
of $TM$).

Define $S^{-\tau}$ to be the fiberwise functional dual of
$S^\tau$. Alternatively, one can define an unstable version of
$S^{-\tau}$ as follows: equivariantly embed $M$ in a
$G$-representation $V$ and let $\nu$ denote its normal bundle. Its
fiberwise one point compactification $S^\nu$ then represents
$S^{-\tau}$ up to a degree shift by $V$, i.e.,
$$
S^{-\tau} \simeq S^{\nu - V}\, .
$$
We call $S^{-\tau}$ the {\it orientation bundle} of $M$.

If ${\cal E}$ is a fibrant and cofibrant $G$-spectrum
over $M$, we set
$$
{}^{-\tau}{\!\cal E} \,\, := \,\, S^{-\tau} \smsh_M {\cal E} \, .
$$
Using the diagonal action, this is a fibred $G$-spectrum over $M$,
called the {\it twist} of ${\cal E}$ by the orientation bundle.

\begin{rem} The reader may object to this construction since
we haven't defined internal smash products of parametrized $G$-spectra.  An
{\it ad hoc} way to define ${}^{-\tau}{\!\cal E}$ is to use the normal bundle
$\nu$. Let ${}^\nu{\!\cal E}$ be the parametrized $G$-spectrum
given by ${}^\nu{\!\cal E}_W = S^\nu \smsh_M {\cal E}_W$,
where we are using the fiberwise smash product in $R(M;G)$.
Then ${}^{-\tau}{\!\cal E}$
can be defined as the parametrized $G$-spectrum 
whose  $W$-th space is  $\Omega_M^V {}^\nu{\!\cal E}_W$.

Alternatively, the reader is referred to \cite[Ch.\ 13]{May-Sig} for the
construction of the internal smash product.
\end{rem}

\subsection*{Poincar\'e duality}

The following is a special case of \cite[th.\ 4.9]{PoHu} and also
a special case of \cite[th.\ 19.6.1]{May-Sig}.

\begin{thm}[``Fiberwise Poincar\'e duality''] 
\label{duality} Let ${\cal E}$ be a
$G$-spectrum over a closed smooth $G$-manifold $M$.
Then there is an isomorphism
$$
h_\bullet^G (M;{}^{-\tau}{\!\cal E})\,\,  \cong \,\, 
h^\bullet_G(M;{\cal E})\, .
$$

\end{thm}

\begin{rems} (1). Here it is essential that ${\cal E}$ be indexed
over a complete $G$-universe.
\smallskip

{\flushleft (2).} Here is how to recover Theorem \ref{duality} from
 \cite[th.\ 4.9]{PoHu}. Take $f\: M \to *$ to be the constant map
to a point. Then, using the notation of 
Remark \ref{Po_notation},  we have an equivalence of $G$-spectra
$$
f_\sharp (\cal E \smsh_M C_f^{-1}) \,\, \simeq \,\,
f_* {\cal E} \, ,
$$
where $C_f^{-1}$ is  the orientation bundle $S^{-\tau}$.
Hence,
$$
f_\sharp {}^{-\tau\!}\cal E \,\, \simeq \,\,
f_* {\cal E} \, .
$$
Now take the equivariant homotopy groups of both sides
and use Remark \ref{Po_notation} to
obtain Theorem \ref{duality}.
\smallskip

{\flushleft (3).} Here is how to recover Theorem \ref{duality} 
from \cite[th.\ 19.6.1]{May-Sig}. 
Using their notation, take $M = E = B$ and
$J = {}^{-\tau}{\cal E}$.  Then one has an 
equivalence of equivariant fibred $G$-spectra over $M$
$$
J\,\, \simeq \,\,  S_p \triangleright (J\smsh_M {\Bbb P_M} S^{\tau})\, .
$$
After applying homology $h_\bullet^G(M;{-})$ to both sides,
the left side becomes, in our notation, $h_\bullet^G(M;{}^{-\tau}{\cal E})$,
whereas the right side, after some unraveling of definitions
and rewriting, becomes $h^\bullet_G(M;{\cal E})$.
\end{rems}

\section{The equivariant complement formula \label{complement-section}}

As in the introduction, let $N$ be a $G$-manifold and let
$i\:Q\subset N$ be a closed $G$-submanifold. Then $N-Q \to N$
is an object of $T(N;G)$. Let
$$
S_N (N-Q) \in T(N;G)
$$
denote its fiberwise suspension. This has the equivariant homotopy type,
of the complement of $Q$ in $N \times [0,1]$. 

Let 
$\nu$ denote the normal bundle of $Q$ in $N$.
We let $D(\nu)$ be its unit disk bundle and $S(\nu)$ its
unit sphere bundle.

\begin{lem} \label{complement} There is an equivariant weak equivalence
$$
D(\nu) \cup_{S(\nu)} N \,\, \simeq \,\, S_N (N-Q)\, .
$$
\end{lem}

\begin{proof} Identify $D(\nu)$ with a closed equivariant tubular neighborhood
of $Q$. Then we have an equivariant pushout
$$
\xymatrix{
S(\nu)  \ar[r] \ar[d] & N- \interior D(\nu)\ar[d]\\
D(\nu) \ar[r] & N\, ,
}
$$
where $\interior D(\nu)$ is identified with the
interior of the tubular neighborhood
and the inclusion  $N-Q \subset N-\interior D(\nu)$ is
an equivariant equivalence.

So we have an equivariant homeomorphism
$$
D(\nu) \cup_{S(\nu)} N \cong N \cup_{N-\interior D(\nu)} N
$$
and the right side has the equivariant homotopy type
of $S_N (N-Q)$, considered as an object of  $R(N;G)$.
\end{proof}

The object $D(\nu) \cup_{S(\nu)} N$ is called the {\it fiberwise 
equivariant Thom 
space} of $\nu$ over $N$. We denote it by $$T_N(\nu)\, .$$ More generally,
for $\xi$ a virtual $G$-bundle over $Q$, one has a {\it fiberwise
equivariant Thom spectrum} $T_N(\xi)$ over $N$.

The virtual $G$-vector bundle 
over $Q$ defined by $i^*\tau_N - \tau_Q$. 
is represented unstably by $\nu$. 
Substituting this and taking fiberwise
suspension spectra of the right side  of Lemma \ref{complement},
we obtain

\begin{cor}[``Complement Formula''] \label{comp-form}
There is an weak equivalence of $G$-spectra over $N$
$$
T_N(i^*\tau_N - \tau_Q) \,\, \simeq \,\, 
\Sigma^{\infty,G}_N S_N(N-Q)\, .
$$
\end{cor}

Note 
the left side of Corollary 
\ref{comp-form} depends only on the underlying homotopy
class of the map $i\: Q \to N$.

\section{The proof of Theorem \ref{equi-int} and Addendum \ref{boundary}
\label{proof-equi-int}}

\begin{proof}[Proof of Theorem \ref{equi-int}]
Consider the equivariant intersection problem 
$$
\SelectTips{cm}{}
\xymatrix{
& N - Q \ar[d] \\
P \ar[r]_f\ar@{..>}[ur]
& N\, ,
}
$$
from \S1. Recall $E \to N$  is  the effect of converting
$N-Q \to N$ into a fibration. 

Consider 
$$
e(\id_P,f^*E) \in H^0_G(P;f^*{\cal E}(i_Q)) \, .
$$
An unraveling of definitions shows
$f^*{\cal E}(i_Q)$ is weak equivalent to the 
naive fiberwise suspension spectrum of $S_P f^*E$. 

Since the object $S_P f^*E$ 
is $(f^*(i_Q)_{!}\cd_\bullet(i_Q)-1)$-connected 
(cf.\ Lemma \ref{connectivity}), 
by Proposition \ref{fib-equi-versus-naive} the map
$$
Q_PS_P f^* E \to Q_P^G S_P f^*E
$$
is $s_\bullet$-connected, where
$$
s_H = \inf_{K\subsetneq H}  f^*(i_Q)_{!}\cd_K(i_Q) - 1 \, .
$$
Using Corollary \ref{factorize-homotopy}, we infer that the 
evident homomorphism
$$
 H^0_G(P;{\cal E}(i_Q)) \cong H^0_G(P;\Sigma^\infty_P S_P f^*E)
\to 
h^0_G(P;{\Sigma^\infty,G}_P S_P f^*E)
$$
from the naive theory to the complete one
is injective when
$$
p^\bullet < s_\bullet \, .
$$
In this range, it follows the image of $e(\id_P,f^*E)$
in $h^0_G(P;\Sigma^{\infty,G}_P f^*E)$ is trivial 
if and only if $e(\id_P,f^*E)$ was trivial to begin with.

The next step is to identify 
$h^\bullet_G(P;\Sigma^{\infty,G}_P S_P f^*E)$.
Using Lemma \ref{complement}, there is a weak equivalence
of objects 
$$
S_N (N-Q)\,\, \simeq \,\, T_N(\nu) \in R(N;G)
$$
where $\nu$ is the normal bundle
of $Q$ in $N$. Consequently, there is an isomorphism
$$
h^\bullet_G(P;\Sigma^{\infty,G}_P S_P f^*E) \,\, \cong \,\,   
h^\bullet_G(P;\Sigma^{\infty,G}_P f^* T_N(\nu)) \, .
$$
By Theorem \ref{duality} the group on the right
is naturally isomorphic to
$$
h_\bullet^G(P;{}^{-\tau_P} \Sigma^{\infty,G}_P f^*T_N (\nu))\, .
$$
An unraveling of the construction  shows
the latter coincides with
the equivariant homotopy groups of the equivariant
Thom spectrum of the virtual bundle $\xi$ over $E(f,i_Q)$ appearing
in the introduction. In particular,
$$
\Omega^G_0(E(f,i_Q);\xi) \cong h_0^G(P;{}^{-\tau_P} 
\Sigma^{\infty,G}_P f^*T_N (\nu) )\, .
$$
With respect to these identifications, we define the {\it equivariant
stable homotopy Euler characteristic}
$$
\chi_G(f) \in \Omega^G_0(E(f,i_Q);\xi) 
$$
to be the unique element that corresponds to
$e(\id_P,f^*E)$.
By the above and Theorem \ref{cohom},
$\chi_G(f)$ fulfills the statement of Theorem \ref{equi-int}.
\end{proof}

\begin{proof}[Proof of Addendum \ref{boundary}]
When $N$ has a boundary and $P$ is closed, the above proof 
extends without modification.
When $P$ has a boundary, one only needs to replace 
Poincar\'e duality (\ref{duality}) in the
closed case with a version of Poincar\'e duality for manifolds
with boundary.

To formulate this, 
let $(M,\partial M)$ is a compact smooth manifold with boundary. Then
duality in this case gives an isomorphism 
$$
h_\bullet^G(M;{}^{-\tau_M}\cal E) \,\, \cong \,\, 
h^\bullet_G(M,\partial M;\cal E) \, .
$$
The right side is defined as follows: for $\alpha = V -W$ and ${\cal E}$ fibrant and
cofibrant, define
$$
h^\bullet_G(M,\partial M;\cal E) \,\, := \,\,
[\Sigma^W_M (M/\!\!/\partial M),\cal E_V]_{R(M;G)}
$$
where $M/\!\!/\partial M$ is the double 
$$
M\cup_{\partial M} M \in R(M;G)
$$
(the section $M\to  M/\!\!/\partial M$ is 
defined using the left summand).
\end{proof}

\section{The proof of Theorems \ref{local} and \ref{global-to-local}
\label{proof-local}}

\begin{proof}[Proof of Theorem \ref{local}]
Recall the factorization $(N-Q) \to E \to N$ in which
$(N-Q) \to E$ is an acyclic cofibration and $E\to N$ is a fibration
of $T(N;G)$.

By construction we have an isomorphism 
$$
H^0_G(X;\Sigma^\infty_N S_N E) \,\, \cong \,\, [X^+,Q_N S_N E]_{R(B;G)}
$$
and an isomorphism
$$
h^0_G(X;\Sigma^{\infty,G}_N S_N E) \,\, \cong \,\, 
[X^+,Q^G_N S_N E]_{R(B;G)}
$$
for any object $X \in T(B;G)$.

With respect to these identifications,
the homomorphism 
$H^0_G(X;\Sigma^\infty_N S_N E)\to 
h^0_G(X;\Sigma^{\infty,G}_N S_N E)$ arises from the map
$$
Q_N S_N E \to Q_N^G S_N E 
$$
by applying homotopy classes $[X^+,{-}]_{R(B;G)}$.

Consider the commutative diagram of abelian groups
\begin{equation} \label{big_diagram}
\xymatrix{
[P_i/\!\!/P_{i-1}, Q_N S_N E]_{R(N;G)} \ar[r]^{j_1} \ar[d]_{\ell_1} &
[P_i^+,Q_N S_N E]_{R(N;G)} \ar[r]^{k_1} 
\ar[d]^{\ell_2} & [P_{i-1}^+,Q_N S_N E]_{R(N;G)} 
\ar[d]^{\ell_3} \\
[P_i/\!\!/P_{i-1}, Q^G_N S_N E]_{R(N;G)} \ar[r]_{j_2} &
[P_i^+,Q^G_N S_N E]_{R(N;G)} \ar[r]_{k_2} & 
[P_{i-1}^+,Q^G_N S_N E]_{R(N;G)}
}
\end{equation}
with exact rows, where the object $P_i/\!\!/P_{i-1}$ is given
by $P_i \cup_{P_{i-1}} N$.

\begin{defn} Let $f_i\: P_i \to N$ be the restriction of $f$ to $P_i$.
Let 
$$
e^i_G(f) \in  [P_i/\!\!/P_{i-1},Q_N S_N E]_{R(N;G)}
$$
be the class determined by the composite
$$
\begin{CD} 
P_i @> f_i >> N @> s_+ >> \Sigma_N E 
\end{CD}
$$
together with the observation that its restriction
to $P_{i-1}$ has a preferred homotopy over $N$ to the 
composite
$$
\begin{CD}
P_{i-1} @> f_{i-1} >> N @> s_- >> \Sigma_N E \, .
\end{CD}
$$
\end{defn}

By essentially the same argument which proves 
Theorem \ref{cohom}, the map $f_i$ is equivariantly homotopic
to a map whose image is disjoint from $Q$, relative to
$P_{i-1}$,  provided 
\begin{itemize}
\item $e^i_G(f) = 0$ and
\item $p^H \le 2f^*(i_Q)_{!}\cd_H(i_Q)  - 3$
for all $(H) \in {\cal I}(G;P)$.
\end{itemize}
If this is indeed the case, the equivariant homotopy extension
property can be used to obtain a new $G$-map $f'$, coinciding
with $f$ on $P_{i-1}$, and satisfying 
$f'(P_i) \cap Q = \emptyset$.

In order to complete the proof of  Theorem \ref{local} 
we will apply a version of Poincar\'e duality. 
Set $H = H_i$ and $P^H_s = P^H- P_H$.
Then the inclusion of pairs
$$
(G\cdot P^H,G\cdot P^H_s) \to (P_i,P_{i-1}) 
$$
is a relative $G$-homeomorphism. Recall the Weyl group
$W(H)$ acts on $P^H$ and restricts to a free action on $P_H$.
The following result follows from the existence of
equivariant tubular neigbhorhoods.

\begin{lem}[{\cite[\S IV]{Davis}}] \label{compactify} 
The open $W(H)$-manifold $P_H$ is the interior of 
a compact free $W(H)$-manifold $\bar P_H$ with
corners. Furthermore, the inclusion $P_H \subset \bar P_H$ is an
equivariant weak equivalence.
\end{lem}

Consider the left
square of diagram  \ref{big_diagram}.  By Lemma
\ref{compactify}, and ``change of groups''
it maps to the square
\begin{equation} \label{little-diagram}
\xymatrix{
[\bar P_H/\!\!/\partial \bar P_H, Q_N S_N E]_{R(N;W(H))} 
\ar[r]^{j'_1} \ar[d]_{\ell'_1} &
[(P^H)^+,Q_N S_N E]_{R(N;W(H))} 
\ar[d]^{\ell'_2}\phantom{\, .} \\
[\bar P_H/\!\!/\partial \bar P_H, Q^G_N S_N E]_{R(N;W(H))} \ar[r]_{j'_2}
 &
[(P^H)^+,Q^G_N S_N E]_{R(N;W(H))} \, .
}
\end{equation}
The proof of Theorem \ref{local} is completed in two steps.

\subsection*{Step 1} The homomorphism $\ell'_1$ is an isomorphism, since
$W(H)$ acts freely on $\bar P_H/\!\!/\partial \bar P_H$ in the
``based'' sense.  This can be proved by an induction argument using
an equivariant cell decomposition, together with the observation 
that the map $Q_N S_N E \to Q_N^G S_N E$ is a weak homotopy equivalence of
underlying topological spaces.

\subsection*{Step 2} There is a relative $W(H)$-homeomorphism
$$
(\bar P_H,\partial \bar P_H) \cong (P^H,P^H_s)
$$
which, together with change of groups, gives an
isomorphism 
$$
[\bar P_H/\!\!/\partial \bar P_H, Q_N S_N E]_{R(N;W(H))} \,\, \cong\,\, 
[P_i/\!\!/P_{i-1}, Q_N S_N E]_{R(N;G)} \, .
$$
We will consider $e^i_G(f)$ to be an element of the left
hand side. Then $\ell'_1(e^i_G(f))$ can be regarded as
an element of  relative cohomology group
$$
h^0_{W(H)}(\bar P_H,\partial \bar P_H; \Sigma^{\infty,G}_N S_N E) \, .
$$
Define $\chi^i_G(f)$ to be its Poincar\'e dual. 
Using the equivariant equivalence $P_H \simeq \bar P_H$,
we  can regard $\chi^i_G(f)$ as living in the 
homology group  
$$
h_0^{W(H)}(P_H; {}^{-\tau_{P_H}\!\!}
(\Sigma^{\infty,G}_{P_H} S_{P_H} f_H^*E))\, .
$$ 
As in the proof of Theorem \ref{equi-int} this homology group is
isomorphic to the equivariant bordism group
$$
\Omega^{W(H)}_0(E(f_{H},i_Q);{}_H\xi)\, .
$$
\end{proof}

\begin{proof}[Proof of Theorem \ref{global-to-local}]
  The proof uses diagrams \eqref{big_diagram} and
\eqref{little-diagram}. The homomorphism
$(t_H)_*$ is identified with 
the Poincar\'e dual of the homomorphism
$j'_2$ of diagram \eqref{little-diagram}. 
Therefore $(t_H)_*$ is injective if and only if $j'_2$ is. 
Let 
$$
\ell\: H^0_G(P;\Sigma^\infty_N S_N E)
\to
h^0_G(P;\Sigma^{\infty,G}_N S_N E)
$$
be the canonical homomorphism. Recall $\chi_G(f)$
is the Poincar\'e dual of $\ell(e_G(f))$.

The class $j'_1(e^i_G(f))$ is the one
associated with  the composition
$$
\begin{CD}
(P^H)^+ \subset P^+ @> f >> N @> s >> S_N E \, ,
\end{CD}
$$ 
i.e., the restriction of $e_G(f)$ to $P^H$.
By hypothesis, $\chi_G(f)$ is trivial, so 
$\ell'_2j'_1(e^i_G(f))$ must also be trivial 
since the latter is the restriction to $P^H$ 
of the trivial class $\ell(e_G(f))$. 

Hence 
$$
j_2' \ell_1'(e^i_G(f))= \ell'_2 j'_1(e^i_G(f)) = 0\, .
$$
Furthermore, since
$j_2'$ is identified with $(t_H)_*$, and the latter
is by hypothesis injective, the vanishing of
$j_2' \ell_1'(e^i_G(f))$ implies $\ell_1'(e^i_G(f)) = 0$. 
Hence, $\chi^i_G(f)$ vanishes too, as it is the Poincar\'e dual of 
$\ell_1'(e^i_G(f))$.
The result is now concluded by induction on $i$ and Theorem \ref{local}.
\end{proof}

\section{The proof of Theorem \ref{lefschetz} \label{proof-lefschetz}}

Given a closed smooth $G$-manifold $M$, we have a commutative square
of equivariant mapping spaces
\begin{equation}\label{fixed_square}
\xymatrix{
\self^\flat(M)^G \ar[r]^\subset \ar[d] & \self(M)^G \ar[d] \\
\maps (M,M \times M - \Delta)^G \ar[r] & \maps (M,M \times M)^G
}
\end{equation}
where 
\begin{itemize}
\item $\Delta := \Delta_M \subset M \times M$ is the diagonal,
\item $M \times M$ is given the diagonal $G$-action,
\item $\self(M)^G$ is the space of equivariant self maps of $M$,
\item  $\self^\flat(M)^G$ is the subspace of fixed point free
equivariant self maps, and
\item the vertical maps of the square are given by taking 
graphs and the horizontal ones are inclusions.
\end{itemize}

\begin{lem}\label{cartesian}
The  square \eqref{fixed_square} is $\infty$-cartesian,
i.e., it is a homotopy pullback.
\end{lem}

\begin{proof} The following idea is used in the proof. Suppose
$X\to Y$ is a map of fibrations over $B$.
Let $X_b$ be the fiber of $X \to B$ at $b\in B$ and similarly
let $Y_b$ be the fiber of $Y \to B$. Then the diagram
$$ \xymatrix{ X_b \ar[r]\ar[d] & Y_b \ar[d]\\ X \ar[r] & Y }
$$
is $\infty$-cartesian.
 
We claim the first factor projection map
\begin{equation}\label{project1}
M\times M -\Delta \to M
\end{equation}
is a fibration of $T(G)$. For, let $H \subset G$
be a subgroup. Taking the induced map of 
fixed point spaces yields the projection map
\begin{equation} \label{project2}
M^H\times M^H -\Delta_{M^H} \to M^H\, .
\end{equation}
Since $M^H$ is a manifold, 
the map \eqref{project2} is a Serre fibration of spaces. 
It follows that the map $\eqref{project1}$ is a fibration
of $T(G)$.

Applying the functor $\maps(M,{-})^G$ to the projection
map, we infer
$$
\maps(M,M\times M - \Delta)^G \to \maps(M,M)^G
$$
is a fibration whose fiber at the identity map
of $M$ is $\maps^\flat(M)^G$.

Similarly, the first factor projection $M \times M \to M$ is an equivariant fibration,
so the induced map 
$$
\maps(M,M\times M)^G \to \maps(M,M)^G
$$
is a fibration whose fiber at the identity is $\maps(M,M)^G$.

It now follows easily from the first paragraph of 
the proof that the square \eqref{fixed_square}
is $\infty$-cartesian.
\end{proof}

From  Lemma \ref{cartesian}, the obstruction to deforming 
an equivariant self map 
$$
f\: M \to M
$$
to a fixed point free map coincides with 
equivariantly deforming its graph 
$\Gamma_f\: M \to M \times M$ off of the diagonal.

Consequently, we are reduced to the equivariant intersection problem
$$
\SelectTips{cm}{}
\xymatrix{
& M \times M - \Delta \ar[d] \\
M \ar[r]_{\Gamma_f} \ar@{..>}[ur] & M \times M\, .
}
$$

We will prove Theorem \ref{lefschetz} using Corollary \ref{descent}.
We will need to compute the codimension function the
diagonal.

\begin{lem} \label{codim-Lefschetz} Let $i_\Delta\: \Delta\subset M \times M$ be the 
inclusion. For $(H) \in {\cal I}(M;G)$, we have
$$
\cd_H(i_\Delta) = m^H \, .
$$
\end{lem}

\begin{proof} If $x\in \Delta  = M$ is a point then
the codimension of the diagonal inclusion 
$M^H_{(x)} \subset M^H_{(x)} \times M^H_{(x)}$ is clearly $m^H(x)$.
\end{proof}

By Lemma \ref{codim-Lefschetz}, the inequality
of Corollary \ref{descent} amounts to the condition
$$
m^H \le 2(\Gamma_f)^*(i_\Delta)_!m^H - 3\, .
$$
By a straightforward argument which we omit, $(\Gamma_f)^*(i_\Delta)_!m^H$
coincides with $m^H$, so the inequality becomes
$$
m^H \ge 3
$$
for $(H) \in {\cal I}(M;G)$.

We now turn to the problem of deciding when the homomorphisms $(t_H)_*$
are injective. What is special about the fixed point case is
that the virtual $W(H)$-bundle ${}^H\xi$, which sits over the space
$$
M^H \times_M L_f M\, ,
$$
is represented by an actual vector bundle. This vector
bundle is just the pullback of the normal bundle 
of the embedding $M^H \subset M$ along the 
(projection) map $M^H \times_M L_f M\to M^H$. Henceforth, 
we identify  ${}^H\!\xi$ with this vector bundle.

Therefore  $(t_H)_*$
is identified with the homomorphism of equivariant framed
bordism groups
\begin{equation} \label{homo}
\Omega^{W(H)}_0(M_H \times_M L_f M;{}_H\xi) \to 
\Omega^{W(H)}_0(M^H \times_M L_f M;{}^H\!\xi) 
\end{equation}
induced by the inclusion 
$$
t_H\: M_H \times_M L_f M \to M^H \times_M L_f M\, ,
$$
where both $\xi_H$ and $\xi^H$ are $W(H)$-vector bundles
and the pullback $t_H^*\xi^H$ is isomorphic to $\xi_H$.

Hence, $(t_H)_*$
arises by taking the $W(H)$-fixed spectra
of the map of equivariant suspension spectra
\begin{equation} \label{susp-map}
\Sigma^\infty_{W(H)}(M_H \times_M L_f M)^{{}_H\xi} \to
\Sigma^\infty_{W(H)}(M^H \times_M L_f M)^{{}^H\!\xi}
\end{equation}
and then applying $\pi_0$.  

The inclusion $M_H \subset M^H$ is $1$-connected, since by hypothesis
$M^H_s := M^H - M_H$ has codimension at least two in $M^H$. Consequently,
the inclusion $M_H \times_M L_f M \to M^H \times_M L_f M$ is also 
$1$-connected. Furthermore,  $M_H \times_M L_f M$ is $W(H)$-free.

If we apply the tom Dieck splitting to the $W(H)$-fixed points of the map
\eqref{susp-map}, we obtain maps of summands of the form
\begin{equation} \label{summand}
\Sigma^\infty ((M_H \times_M L_f M)^{{}_H\xi})^K_{hW'(K)}
\to 
\Sigma^\infty ((M^H \times_M L_f M)^{{}^H\!\xi})^K_{hW'(K)} \, ,
\end{equation}
where $(K)$ ranges through the conjugacy classes of subgroups
of $W(H)$, $W'(K)$ denotes the Weyl group of $K$ in $W(H)$ and
the subscript ``${}_{hW'(K)}$'' is an abbreviation for the
Borel construction (in the notation above, we are first Thomifying,
then taking fixed points and thereafter taking the Borel construction).

If $K$ is not the trivial group, then the freeness of the action
implies the domain of \eqref{summand} is contractible, 
and therefore this map induces an injection on $\pi_0$.
If $K$ is trivial, then the map takes the form
$$
\Sigma^\infty (M_H \times_M L_f M)^{{}_H\xi}_{hW(H)}
\to 
\Sigma^\infty (M^H \times_M L_f M)^{{}^H\!\xi}_{hW(H)}
$$
which is evidentally $1$-connected. Assemblying these
injections, one sees the homomorphism \eqref{homo} is
also injective. Therefore, the homomorphism
$(t_H)_*$ appearing in the statement of Corollary \ref{descent} is injective
for every $(H) \in {\cal I}(M;G)$.

The proof of
Theorem \ref{lefschetz} is now completed by applying Corollary
\ref{descent}.

\section{The proof of Theorems \ref{periodic} and \ref{stalk} \label{proof-periodic}}

Let $f\: M \to M$ be a self map of a closed smooth manifold $M$.
The {\it Fuller map} of $f$ is the 
$\Bbb Z_n$-equivariant self map of $M^{\times n}$ given by
$$
(x_1,...,x_n) \mapsto (f(x_{n}),f(x_1),f(x_2),...,f(x_{n-1}))
$$
(compare Fuller \cite{Fuller}). Here $n \ge 2$ and $\Bbb Z_n$ acts by 
cyclic permutation of factors. The 
assignment $x \mapsto (x,f(x),\dots,f^{n-1}(x))$ defines
a $\Bbb Z_n$-equivariant bijective correspondence 
between the $n$-periodic point set of $f$ and the fixed point
set of $\Phi_n(f)$. In particular, $f$ is $n$-periodic point free
if and only if $\Phi_n(f)$ is fixed point free. We wish to know
whether this statement is true up to homotopy. 

Let $\self(M)$ be the space of self maps of $M$,
and $\self(M^{\times n})^{\Bbb Z_n}$  the
space of $\Bbb Z_n$-equivariant self maps of $M^{\times n}$. The
{\it Fuller transform} 
$$
\Phi_n\: \self(M) \to \self(M^{\times n})^{\Bbb Z_n}
$$
is defined by $f\mapsto \Phi_n(f)$. 

Let 
$$
\self^\flat_n(M) \subset \self(M)
$$
be the subspace of self maps having no $n$-periodic points.
Let
$$
\text{end}^{\flat}(M^{\times n})^{\Bbb Z^n} 
\subset \text{end}(M^{\times n})^{\Bbb Z_n}
$$
be the subspace of equivariant 
self maps of $M^{\times n}$ which are fixed point free.

Then there is a commutative diagram of spaces
\begin{equation} \label{period_diagram}
\xymatrix{
\text{end}^{\flat_n}(M) 
\ar[r]\ar[d] & \text{end}(M)\ar[d] \\
\text{end}^{\flat}(M^{\times n})^{\Bbb Z_n} 
\ar[r] & \text{end}(M^{\times n})^{\Bbb Z_n}
}
\end{equation}
where the vertical maps are given by the Fuller transform
and the horizontal ones are inclusions. The square is
cartesian, i.e., it is a pullback. We wish to understand the
extent to which it is a homotopy pullback.

\begin{ques}  Is the above square  
$0$-cartesian?
\end{ques}

That is, is the map from $\text{end}^{\flat_n}(M)$
to the corresponding homotopy pullback a surjection on 
components? If yes, it would reduce the problem of studying the 
$n$-periodic points of $f$ to the $\Bbb Z_n$-equivariant
fixed point theory of $\Phi_n(f)$. At the time of
writing we do not know this to be the case. Nevertheless,
we can still use the diagram to get an invariant of 
self maps which is trivial when the self map is homotopic
to an $n$-periodic point free one.

\begin{defn} Set
$$
\ell_n(f) := \ell_{\Bbb Z_n}(\Phi_n(f))\, .
$$
\end{defn}
By a straightforward 
calculation we omit, $\ell_n(f)$ lives in the group
$$
\Omega_0^{{\Bbb Z}_n, \text{fr}}(\hoP_n(f))
$$
appearing in the statement of Theorem \ref{periodic}.
It is clear that $\ell_n(f)$ vanishes when $f$ is homotopic
to an $n$-periodic point free map. If we apply
Theorem \ref{lefschetz} to $\Phi_n(f)$ we obtain

\begin{cor} Assume $\dim M \ge 3$ and $\ell_n(f) = 0$.
Then $\Phi_n(f)$ is equivariantly homotopic to a fixed point
free map.
\end{cor}

As mentioned in \S1, a result of Jezierski
\cite{Jezierski}
asserts $f$ is homotopic to an $n$-periodic point free  map
if $\dim M \ge 3$ and the Nielsen numbers $N(f^k)$ vanish
for each $k$ a divisor of $n$. Conjecturally,
the invariant $\ell_n(f)$ contains at least as much information
as these Nielsen numbers (additional 
evidence for this is provided below
in Theorem \ref{stalk}). If one assumes this to be the case,
then Jezierski's theorem tends to suggest that
the square \eqref{period_diagram} is $0$-cartesian.
However, we do not see any homotopy 
theoretic reason why that should be true.

\begin{proof}[Proof of Theorem \ref{stalk}] The tom Dieck splitting yields a
decomposition of $\Omega_0^{\Bbb Z_n,\text{\rm fr}}(\hoP_n(f))$ into
summands of the form
$$
\Omega_0^{\text{\rm fr}}(E\Bbb Z_k \times_{\Bbb Z_k}\hoP_k(f))
$$
for $k$ a divisor of $n$, where we are using the fact 
$\hoP_n(f)^{\Bbb Z_k} = \hoP_k(f)$.

Since the zeroth framed bordism of a space is the free abelian
group on its path components,
it will suffice to show $\pi_0(E\Bbb Z_k \times_{\Bbb Z_k}\hoP_k(f))$
is isomorphic to $\pi_{\rho,k}$. We first compute 
the set of components of $\hoP_k(f)$. 

Recall from \S1 that a point of  $\hoP_k(f)$ is given by
a $k$-tuple of paths
$$
(\lambda_1,...,\lambda_k)
$$
subject to the constraint $f(\lambda_{i+1}(0)) = \lambda_i(1)$
where $i$ is taken modulo $k$.
Two points $(\lambda_1,...,\lambda_k)$ and 
$(\gamma_1,...,\gamma_k)$ are in the same component if and only if
there are paths $\alpha_i$ having
initial point $\lambda_i(0)$ and terminal point $\gamma_i(0)$
such that the concatenated paths
$$
f(\alpha_i)\ast \lambda_{i+1} \quad \text{and } \quad
\gamma_{i+1}\ast \alpha_{i+1}
$$
are homotopic relative to their endpoints, for $i = 1,2,\dots,k$.

Since $M$ is connected, each component of $\hoP_k(f)$ has a point of the form
$(\lambda_1,\lambda_2,\dots,\lambda_k)$ satisfying $\lambda_i(0) = *$ and
$\lambda_i(1) = f(*)$. Let $\pi_f$ denote the set of homotopy classes of paths
in $M$ joining the basepoint $*$ to the point $f(*)$. 
A choice of element $[\alpha]$ of 
$\pi_f$ determines an isomorphism with $\pi$.
Using this isomorphism the set of path components of 
$\hoP_k(f)$ is a quotient of the $k$-fold cartesian product
$$
\pi \times \cdots \times \pi
$$
with respect to the equivalence relation
$$
(x_1,x_2,\dots,x_k) \, \sim \, 
(g_1x_1\rho(g_2)^{-1},g_2x_2\rho(g_3)^{-1},\dots,g_kx_k\rho(g_1)^{-1})
$$
for $x_i,g_i \in \pi$. Using this relation, the $k$-tuple
$(x_1,\dots x_k)$ is equivalent to the $k$-tuple
$$
(y,1,\dots 1)
$$
where $y = x_1\rho(x_2)\rho^2(x_3)\cdots \rho^{k-1}(x_k)$.
Furthermore, any two elements of the form $(y,1,\dots,1)$ and
$(z,1,\dots,1)$ are related precisely when $z = gy\rho^k(g)^{-1}$
for some element $g\in \pi$. Summarizing thus far, we have shown
$\pi_0(\hoP_k(f))$ is the quotient of $\pi$ with respect
to the equivalence relation
$$
y \, \sim\, g y \rho^k(g)^{-1}
$$
for $g,y\in \pi$.

To complete the proof of Theorem \ref{stalk}, one notes the set
of path components of the Borel construction coincides with the
coinvariants of $\Bbb Z_k$ acting on $\pi_0(\hoP_k(f))$.  With respect
to the $k$-tuple description of $\pi_0(\hoP_k(f))$, the action is 
induced by cyclic permutation of factors: $(x_1,x_2,\dots,x_k) \mapsto
(x_k,x_1,\dots,x_{k-1})$.  If we identify this element with
$(y,1,\dots 1)$ with $y$ as above, then the result of acting by a
generator of the cyclic group results in an element equivalent to
$(\rho(y),1,\dots 1)$.  Consequently, $\pi_0(E\Bbb Z_k \times_{\Bbb
Z_k}\hoP_k(f))$ is obtained from $\pi_0(\hoP_k(f))$ by imposing the
additional relation $y \sim \rho(y)$.  Hence, the set of path
components of the Borel construction is isomorphic to $\pi_{\rho,k}$.
\end{proof}

\begin{conjecture} Let ${\cal N}_k(f)$ be the number 
of non-zero terms in $\ell_n^k(f)$ expressed as a linear
combination of the basis elements of $\Bbb Z[\pi_{\rho,k}]$.
Then ${\cal N}_k(f)$ equals the Nielsen number of $f^k$.
\end{conjecture}

\end{document}